\documentclass[12pt]{amsart} 
\usepackage[margin=1.85cm]{geometry}

\usepackage{xr} 
\usepackage{hyperref}
\usepackage{appendix}
\usepackage{amsthm}
\usepackage{amsmath}
\usepackage{amssymb}
\usepackage[all]{xy}
\usepackage{graphicx}
\usepackage{indentfirst}
\usepackage{bm}
\usepackage{mathrsfs}
\usepackage{latexsym}
\usepackage{color}
\usepackage{microtype}
\usepackage{tikz}
\usepackage{enumerate}
\usepackage{comment}
\usepackage{braket}
\usepackage{manfnt}
\usepackage{hyperref}
\usepackage[utf8]{inputenc}
\usepackage{relsize}
\usepackage{listings}
\usepackage{textcomp}
\usepackage{spverbatim}
\usepackage[T1]{fontenc}
\usepackage{lmodern}
\usepackage{url}
\usepackage{pdflscape}
\usepackage{enumitem}
\usepackage{etoolbox}

\allowdisplaybreaks

\begin{document}
\theoremstyle{plain}
\newtheorem{theorem}{Theorem}[section]
\newtheorem{main}{Main Theorem}
\newtheorem{proposition}[theorem]{Proposition}
\newtheorem{corollary}[theorem]{Corollary}
\newtheorem{lemma}[theorem]{Lemma}
\newtheorem{conjecture}[theorem]{Conjecture}
\newtheorem{claim}[theorem]{Claim}
\newtheorem{fact}[theorem]{Fact}
\newtheorem{question}[theorem]{Question}
\newtheorem*{st}{Statements}	

\numberwithin{equation}{section}

\theoremstyle{definition}
\newtheorem{definition}[theorem]{Definition}
\newtheorem{notation}[theorem]{Notation}
\newtheorem{convention}[theorem]{Convention}
\newtheorem{example}[theorem]{Example}
\newtheorem{remark}[theorem]{Remark}
\newtheorem*{ac}{Acknowledgements}	

\newcommand{\one}{{\bf 1}} 
\newcommand{\Rep}{\mathrm{Rep}}
\newcommand{\Id}{\mathrm{id}}
\newcommand{\Aut}{\mathrm{Aut}}
\newcommand{\ch}{\mathrm{ch}}
\newcommand{\VVec}{\mathrm{Vec}}
\newcommand{\PSU}{\mathrm{PSU}}
\newcommand{\SO}{\mathrm{SO}}
\newcommand{\SU}{\mathrm{SU}}
\newcommand{\SL}{\mathrm{SL}}
\newcommand{\FPdim}{\mathrm{FPdim}}
\newcommand{\FPdims}{\mathrm{FPdims}}
\newcommand{\Tr}{\mathrm{Tr}}
\newcommand{\ord}{\mathrm{ord}}
\newcommand{\customcite}[2]{\cite[#2]{#1}}
\newcommand{\mC}{\mathcal{C}}
\newcommand{\mB}{B}
\newcommand{\hc}{\hom_{\mC}}
\newcommand{\id}{{\bf 1}} 
\newcommand{\spec}{i_0} 
\newcommand{\Sp}{i_0} 
\newcommand{\SpecS}{I_{s}} 
\newcommand{\SpS}{I_{s}'} 
\newcommand{\field}{\mathbb{K}} 
\newcommand{\white}{\textcolor{red}{\bullet}} 
\newcommand{\black}{\bullet} 
\newcommand{\proofsketch}{ \noindent \textit{Proof sketch.} }
\newenvironment{restatetheorem}[1]
  {
   \par\addvspace{0.25\baselineskip}
   \noindent\textbf{Theorem \ref{#1}.}\ \itshape
  }
  {
   \par\addvspace{0.25\baselineskip}
  }

\newcommand{\sebastien}[1]{\textcolor{blue}{#1 - Sebastien}}

\newcommand{\RR}{\mathbb{R}}
\newcommand{\CC}{\mathbb{C}}
\newcommand{\QQ}{\mathbb{Q}}
\newcommand{\ZZ}{\mathbb{Z}}

\newcommand{\Normaliz}{\textsf{Normaliz}}
\newcommand{\SageMath}{\textsf{SageMath}}
\newcommand{\GAP}{\textsf{GAP}}

\title{Classifying integral Grothendieck rings up to rank 5 and beyond}

\author{Max A. Alekseyev}
\address{M.A.~Alekseyev, Department of Mathematics, George Washington University, Washington, DC, USA}
\email{maxal@gwu.edu}

\author{Winfried Bruns}
\address{W.~Bruns, Institut für Mathematik, Universität Osnabrück, 49069 Osnabrück, Germany}
\email{wbruns@uos.de}

\author{Jingcheng Dong}
\address{J.~Dong, College of Mathematics and Statistics, Nanjing University of Information Science and Technology, Nanjing 210044, China}
\email{jcdong@nuist.edu.cn}

\author{Sebastien Palcoux}
\address{S.~Palcoux, Beijing Institute of Mathematical Sciences and Applications, Huairou District, Beijing, China}
\email{sebastien.palcoux@gmail.com}
\urladdr{https://sites.google.com/view/sebastienpalcoux}

\maketitle

\begin{abstract} 
In this paper, we define a Grothendieck ring as a fusion ring categorifiable into a fusion category over the complex field. An integral fusion ring is called Drinfeld if all its formal codegrees are integers dividing the global Frobenius--Perron dimension. Every integral Grothendieck ring is necessarily Drinfeld.

Using the fact that the formal codegrees of integral Drinfeld rings form an Egyptian fraction summing to 1, we derive a finite list of possible global FPdims for small ranks. Applying Normaliz, we classify all fusion rings with these candidate FPdims, retaining only those admitting a Drinfeld structure. To exclude Drinfeld rings that are not Grothendieck rings, we analyze induction matrices to the Drinfeld center, classified via our new Normaliz feature. Further exclusions and constructions involve group-theoretical fusion categories and Schur multipliers.

Our main result is a complete classification of integral Grothendieck rings up to rank 5, extended to rank 7 in odd-dimensional and noncommutative cases using Frobenius--Schur indicators and Galois theory. Moreover, we show that any noncommutative, odd-dimensional, integral Grothendieck ring of rank at most 22 is pointed of rank 21.

We also classify all integral 1-Frobenius Drinfeld rings of rank 6, identify the first known non-Isaacs integral fusion category (which turns out to be group-theoretical), classify integral noncommutative Drinfeld rings of rank 8, and integral 1-Frobenius MNSD Drinfeld rings of rank 9. Finally, we determine the smallest-rank exotic simple integral fusion rings: rank 4 in general, rank 6 in the Drinfeld case, and rank 7 in the 1-Frobenius Drinfeld case.
\end{abstract}



\tableofcontents

\section*{About the Appendices}

The appendices, which contain extensive fusion data, have been omitted from the published version due to space constraints. They are referred to as~\cite{appendices}.

\section{Introduction}   \label{sec:Intro}
In this paper, all fusion categories are assumed to be over $\mathbb{C}$. 

\subsection{Integral fusion categories in the literature}
The classification of fusion categories has emerged as a central theme in modern algebra, bridging representation theory, operator algebras, and mathematical physics. The overarching goal is to systematically organize these structures, which generalize the representation theories of finite groups and quantum groups. A principal strategy is to classify categories according to the Frobenius--Perron dimensions of their simple objects, yielding a fundamental dichotomy between integral (all dimensions are integers) and non-integral categories. An integral fusion category is pseudo-unitary—that is, its categorical dimension equals its Frobenius-Perron dimension—and therefore spherical, hence pivotal~\cite{EGNO15}.

Integral fusion categories---known to be equivalent to the representation categories of finite-dimensional semisimple quasi-Hopf algebras \cite{EO04}---are studied along two main directions. One approach seeks to classify all integral fusion categories of a fixed finite dimension. Landmark results include complete classifications for dimensions $p$, $p^{2}$, and $pq$ \cite{ENO05,EGO04}, where $p$ and $q$ are distinct primes. The other approach focuses on classification by rank (the number of simple objects), particularly in low-rank cases. 

For instance, integral (more generally, pivotal) fusion categories of ranks $2$ and $3$ have been fully classified \cite{Os03,Os15}, often exhibiting highly constrained fusion rules governed by number-theoretic and combinatorial conditions arising from Frobenius--Perron theory, Frobenius--Schur indicators, and related tools. In contrast, classification efforts stall at rank $4$ due to a sharp increase in complexity. The primary challenges stem from the rapid growth of possible fusion rings and the difficulty of solving the pentagon equations necessary for categorification. Moreover, the classification of integral fusion categories already subsumes that of finite group representations, a problem still out of reach.

Research in both directions remains vibrant, with recent advances increasingly exploiting deeper structural features, such as braidings and the non-degeneracy of the $S$-matrix \cite{CP22,ABPP,CGP24}. Nevertheless, general classification results for integral fusion categories remain scarce. 

This paper extends the classification of the Grothendieck rings of integral fusion categories up to rank~$5$ unconditionally, and to higher ranks under additional assumptions.

\subsection{Egyptian fractions and Drinfeld rings}

A previous work \cite{ABPP}, inspired by \cite{BrRo}, classified the Grothendieck rings (and modular data) of integral modular fusion categories up to rank $13$, relying on the key observation that being half-Frobenius induces an Egyptian fraction with squared denominators. The present paper also employs Egyptian fractions, albeit in a somewhat dual manner.


The concept of a Drinfeld ring is introduced in general in \S\ref{sub:Drin}, but in the integral setting, it simply means that the formal codegrees $(f_i)$—defined in \S\ref{sub:formal}—are integers that divide the global $\FPdim = f_1$. Additionally, in the commutative case, $\sum_i \frac{1}{f_i} = 1$, forming an Egyptian fraction. This concept also extends to the noncommutative case by considering multiplicities (see \S\ref{sub:formal}), but every fusion ring up to rank $5$ is commutative (Proposition \ref{prop:min6}). Note that the Egyptian fractions must satisfy the \emph{divisibility assumption}, meaning each $f_i$ divides $f_1$ (see \S\ref{sub:EgyAlgo}, \cite{A374582} and \cite{MaxScripts}).

There are only a finite number of Egyptian fractions of a given length $\ell$ (see \cite{Lan03} and \cite{A002966}), and a finite number of integral fusion rings for a given global $\FPdim$ and rank $\ell$. Consequently, there are finitely many possible integral Drinfeld rings of a fixed rank, as noted in \cite[Proposition 8.38]{ENO05}. This finiteness makes computer-assisted classification feasible for small ranks. 

\subsection{Computational vs. non-computational} \label{sub:CvsNV}

The aim of this subsection is to distinguish between the computer-assisted components of this work and those that are not. The classifications of Egyptian fractions, integral fusion rings, and induction matrices are fully automated; the corresponding algorithmic details using \SageMath{} and \Normaliz{} are provided in \S\ref{sec:AlgoDetails}. The hardest cases require heavy computations on HPC systems.

Concerning the problem of categorification over $\mathbb{C}$ of a given integral fusion ring: if it is not Drinfeld, or if it is Drinfeld without an induction matrix (Lemmas~\ref{lem:NoInd} and \ref{lem:NoInd2}), then it is excluded, and the process is almost fully automated. Otherwise, additional analysis of the induction matrices is required: if the induction matrices induce an embedding into the Drinfeld center while the fusion ring is known not to be the Grothendieck ring of a premodular category, or if the corresponding type of the Drinfeld center is known not to arise from an integral modular fusion category (Lemmas~\ref{lem:1126}, \ref{lem:AllBraided[1,1,1,3,12]}, \ref{lem:11266}, \ref{lem:112210}), then the integral fusion ring is also excluded, and the process reduces to consulting the literature. Otherwise, more elaborate arguments are required, involving de-equivariantization (Lemma~\ref{lem:11136}) or group-theoretical categories and the Schur multiplier (Remark~\ref{rk:alternative}, \S\ref{sub:schur}, and \S\ref{sub:r7d903}). All remaining fusion rings admit known models: near-group categories, Kac algebras, group-theoretical categories, or a variant of $\Rep(S_4)$. The group-theoretical ones required additional work to be identified (\S\ref{sub:r7d60nc} and \S\ref{sub:NonIsaacs}), partially automated using \GAP{}.  

This approach allowed us to complete the classification of integral Grothendieck rings up to rank~$5$ (Theorem~\ref{thm:main}). We then extended the classification up to rank~$7$, in the noncommutative case (Theorem~\ref{thm:IntroNCGrIntRank7}) and in the odd-dimensional case (Theorem~\ref{thm:mainodd}). In these two situations, the classification of Egyptian fractions can be refined by imposing additional constraints on the formal codegrees: in the noncommutative case they are obtained via Galois-theoretic methods (Proposition~\ref{prop:min6}, Lemmas~\ref{lem:excl} and~\ref{lem:NCmin}), while in the odd-dimensional case, these follow from Frobenius--Schur indicators (Proposition~\ref{prop:MNSDFormalCodegrees}). Finally, by combining these two cases, the classification can be extended further up to rank~$22$ (Theorem~\ref{thm:OddNC2Intro}), relying on additional restrictions (Proposition~\ref{prop:DimDet} and Corollary~\ref{cor:OddNC}).


\subsection{Classification results} \label{sub:IntroClass}

Our classification of integral Drinfeld rings is organized according to several structural constraints: the general case (\S\ref{sec:General} and~\cite[\ref{sec:GeneralA}]{appendices}), the odd-dimensional case (\S\ref{sec:Odd} and~\cite[\ref{sec:OddA}]{appendices}), and the noncommutative case (\S\ref{sec:NC} and~\cite[\ref{sec:NCA}]{appendices}). Within each setting, we further refine the classification either under the $1$-Frobenius condition—meaning that each basic $\FPdim$ divides the global $\FPdim$ (see~\S\ref{sub:Fu})—or under specific bounds on the global $\FPdim$.
Most non-Grothendieck integral Drinfeld rings were detected through an analysis of the possible induction matrices (see~\S\ref{sec:IndMat} and~\S\ref{sec:ExcMod}), using a new feature of \Normaliz{}~\cite[\S H.6.3]{NorManual} developed specifically for this work.

In the general case, this leads to the following result:


\begin{theorem} \label{thm:main}
An integral fusion category up to rank $5$ (over $\mathbb{C}$) is Grothendieck equivalent to one of the following:
\begin{itemize}
\item 	$\Rep(G)$ with $G = C_1, C_2, C_3, C_4, C_5$, $C_2^2$, $S_3$, $S_4$, $D_4$, $D_5$, $D_7$, $F_5$, $C_7 \rtimes C_3$, $A_4$ and $A_5$. 
\item Tambara-Yamagami near-group $C_4+0$, see \cite{TY98},
\item Isotype variant (but non-zesting) of ${\rm Rep}(S_4)$, see \cite[\S 4.4]{LPR22}. 
\end{itemize}
The fusion data are available in \cite[\S\ref{sub:UpToRank5}]{appendices}.
\end{theorem} 

In combination with~\cite[Proposition 9.11 and Theorem 9.16]{ENO11}, we deduce the following:

\begin{corollary} \label{cor:main}
A perfect integral fusion category up to rank $5$ (over $\mathbb{C}$) is equivalent to $\Rep(A_5)$.
\end{corollary}


From the comprehensive classifications in~\S\ref{sec:General} and \cite[\ref{sec:GeneralA}]{appendices}, we find that there are exactly $29$ integral Drinfeld rings of rank at most~$5$, and $58$ integral $1$-Frobenius Drinfeld rings of rank~$6$. Among these, only two are non-pointed and simple: the Grothendieck rings of~$\Rep(G)$ for $G = A_5$ and $G = \mathrm{PSL}(2,7)$.

\begin{corollary} \label{cor:main2} 
A non-pointed simple integral $1$-Frobenius fusion category of rank at most~$6$ (over $\mathbb{C}$) is Grothendieck equivalent to $\Rep(G)$, where $G = A_5$ or $\mathrm{PSL}(2,7)$.
\end{corollary}
The first known example of a non-Isaacs fusion category was identified in~\cite{BP25}: the Extended Haagerup fusion category $\mathcal{EH}_1$, which is non-integral and has rank~$6$. From Theorem \ref{thm:main} and \S\ref{sub:NonIsaacs}, we deduce:
\begin{corollary} \label{cor:IntnonIsaacs}
The smallest rank for a non-Isaacs integral fusion category is $6$, as realized by the group-theoretical fusion category $\mathcal{C}(A_5,1,S_3,1)$.
\end{corollary}

Without assuming the $1$-Frobenius condition, we obtain the following result in the noncommutative setting:

\begin{theorem} \label{thm:IntroNCDrIntRank6}
Up to rank~$6$, the only noncommutative integral Drinfeld ring is the group ring~$\mathbb{Z}S_3$.
\end{theorem}

We completed the classification of the noncommutative integral Grothendieck rings up to rank $7$:

\begin{theorem} \label{thm:IntroNCGrIntRank7}
For ranks up to $7$, an integral fusion category with a noncommutative Grothendieck ring is Grothendieck equivalent to one of the following:
\begin{itemize}
\item Rank $6$, $\FPdim \ 6$,  type $[1,1,1,1,1,1]$: $\VVec(S_3)$;
\item Rank $7$, $\FPdim \ 24$, type $[1,1,1,2,2,2,3]$: $\Rep(H)$ with $H$ the Kac algebra in \cite[Theorem 14.40 (VI)]{IK02};
\item Rank $7$, $\FPdim \ 60$, type $[1,1,1,3,4,4,4]$: group-theoretical $\mathcal{C}(A_5, 1, A_4, 1)$, see \S\ref{sub:r7d60nc}.
\end{itemize}
The fusion data are available in \cite[\ref{sec:NCA}]{appendices}.
\end{theorem}

We also complete the classification of all $29$ noncommutative integral Drinfeld rings of rank at most $8$; see \S\ref{sec:NC} and \cite[\ref{sec:NCA}]{appendices}. There is exactly one such ring of rank $6$ and three of rank $7$, all of which are $1$-Frobenius. At rank~$8$, there are $25$ examples, precisely five of which are not $1$-Frobenius.

%
%
%

Regarding the odd-dimensional case: the Grothendieck ring of any odd-dimensional integral fusion category is an MNSD integral Drinfeld ring (see Definition~\ref{def:MNSDring}). A complete classification up to rank~$7$ yields eight such Drinfeld rings (see \S\ref{sec:Odd} and \cite[\ref{sec:OddA}]{appendices}), all of which are categorifiable except one (see \S\ref{sub:r7d903}). We then obtain the following:

\begin{theorem} \label{thm:mainodd}
Every odd-dimensional integral fusion category of rank at most~$7$ is Grothendieck equivalent to a Tannakian category, namely $\Rep(G)$, where $G = C_1$, $C_3$, $C_5$, $C_7$, $C_7 \rtimes C_3$, $C_{13} \rtimes C_3$, or $C_{11} \rtimes C_5$.
\end{theorem}
Our first known odd-dimensional integral fusion category not Grothendieck equivalent to a Tannakian one appears at rank~$27$; see~\S\ref{sub:NonTann}.

Concerning the 1-Frobenius MNSD integral Drinfeld rings, we classified all such rings of rank~$9$, yielding ten examples described in \cite[\S\ref{sub:Rank9MNSD1Frob}]{appendices}. We also classified all such rings of rank~$11$ with global \( \FPdim \le 10^9 \); see \S\ref{sub:MNSD1FrobR11}. All MNSD integral Drinfeld rings identified so far are commutative. More generally, we established the following result:

\begin{theorem} \label{thm:OddNC2Intro}
Let \( \mathcal{C} \) be a noncommutative, odd-dimensional, integral fusion category. Then the rank of \( \mathcal{C} \) is at least~$21$. If equality holds, then \( \mathcal{C} \) is pointed and Grothendieck equivalent to \( \VVec(C_7 \rtimes C_3) \). Moreover, if \( \mathcal{C} \) is non-pointed, then its rank is at least~$23$.
\end{theorem}


%

\subsection{Exotic Drinfeld rings} \label{sub:ExoticIntro}

Beyond classifications, the main objective of this paper is to highlight intriguing integral Drinfeld rings whose categorification remains unknown. We focus in particular on the smallest simple candidates, with the aim of stimulating further progress on the categorification problem. These candidates merit close attention, both from theoretical and computational perspectives.

We say that a weakly-integral fusion ring is \emph{exotic} if it is not the Grothendieck ring of any weakly group-theoretical fusion category. This terminology reflects the fact that if such an exotic fusion ring were to admit a categorification, it would yield a positive answer to~\cite[Question~2]{ENO11}. The following theorem, extracted from~\cite[\S5]{LPR23}, characterizes exotic simple integral fusion rings:

\begin{theorem}
A simple integral fusion ring is exotic if and only if it is not the character ring of a finite simple group.
\end{theorem}

Theorems~\ref{thm:ExoticRk4}, \ref{thm:mainexotic}, and~\ref{thm:mainexotic2} respectively determine the minimal possible rank of an exotic simple integral fusion ring in the general case, the Drinfeld case, and the $1$-Frobenius Drinfeld case. Each theorem is accompanied by an example realizing the smallest possible $\FPdim$ in its setting.

\begin{theorem}[\cite{BP24}] \label{thm:ExoticRk4}
The smallest possible rank of an exotic simple integral fusion ring is~$4$.
\end{theorem}

We suspect that there are infinitely many simple integral fusion rings of rank~$4$. If so, Proposition~\ref{prop:ext1} would then imply the existence of infinitely many integral fusion rings in every rank~$\ge 4$. There are $1121$ simple integral fusion rings of rank~$4$ with $\FPdim \le 10^9$ (see~\cite{BP24}), but no clear pattern. The one with the smallest global $\FPdim$ is:

\begin{itemize}
\item $\FPdim \ 574$, type $[1, 11, 14, 16]$, duality $[0, 1, 2, 3]$, fusion data:
$$\left[ \begin{smallmatrix}1&0&0&0 \\ 0&1&0&0 \\ 0&0&1&0 \\ 0&0&0&1 \end{smallmatrix} \right],  \ 
  \left[ \begin{smallmatrix} 0&1&0&0 \\ 1&0&4&4 \\ 0&4&1&6 \\ 0&4&6&3 \end{smallmatrix} \right],  \ 
  \left[ \begin{smallmatrix} 0&0&1&0 \\ 0&4&1&6 \\ 1&1&12&1 \\ 0&6&1&9 \end{smallmatrix} \right],  \ 
  \left[ \begin{smallmatrix} 0&0&0&1 \\ 0&4&6&3 \\ 0&6&1&9 \\ 1&3&9&6\end{smallmatrix} \right]$$
\end{itemize}
From the comprehensive classification in~\cite[\S\ref{sub:UpToRank5}]{appendices}, such simple exotic examples of rank at most~$5$ cannot be Drinfeld; in particular, they cannot be categorified as fusion categories over $\mathbb{C}$. However, as shown in~\cite[\S\ref{subsub:non1FrobR6}]{appendices}, we have:

\begin{theorem} \label{thm:mainexotic}
The smallest possible rank of an exotic simple integral Drinfeld ring is~$6$.
\end{theorem}

Regarding the exotic cases referenced in Theorem~\ref{thm:mainexotic}, there are precisely two examples of rank~$6$ and $\FPdim \le 200000$, both of the same type (see~\cite[\S\ref{subsub:non1FrobR6}]{appendices}). We expect that no other example exists at this rank. Among the two, one is not 3-positive (as defined below Theorem \ref{thm:n-positive}), and is therefore excluded from unitary categorification. The other, presented below, is 3-positive.

\begin{itemize}
\item $\FPdim \ 1320$, type $[1, 9, 10, 11, 21, 24]$, duality $[0, 1, 2, 3, 4, 5]$, fusion data:
$$ \normalsize{\left[ \begin{smallmatrix}1 & 0 & 0 & 0 & 0 & 0 \\ 0 & 1 & 0 & 0 & 0 & 0 \\ 0 & 0 & 1 & 0 & 0 & 0 \\ 0 & 0 & 0 & 1 & 0 & 0 \\ 0 & 0 & 0 & 0 & 1 & 0 \\ 0 & 0 & 0 & 0 & 0 & 1 \end{smallmatrix}\right] , \ 
 \left[ \begin{smallmatrix} 0 & 1 & 0 & 0 & 0 & 0 \\ 1 & 0 & 0 & 1 & 1 & 2 \\ 0 & 0 & 1 & 1 & 1 & 2 \\ 0 & 1 & 1 & 1 & 1 & 2 \\ 0 & 1 & 1 & 1 & 3 & 4 \\ 0 & 2 & 2 & 2 & 4 & 3 \end{smallmatrix}\right] , \ 
 \left[ \begin{smallmatrix} 0 & 0 & 1 & 0 & 0 & 0 \\ 0 & 0 & 1 & 1 & 1 & 2 \\ 1 & 1 & 0 & 0 & 2 & 2 \\ 0 & 1 & 0 & 1 & 2 & 2 \\ 0 & 1 & 2 & 2 & 3 & 4 \\ 0 & 2 & 2 & 2 & 4 & 4 \end{smallmatrix}\right] , \ 
 \left[ \begin{smallmatrix} 0 & 0 & 0 & 1 & 0 & 0 \\ 0 & 1 & 1 & 1 & 1 & 2 \\ 0 & 1 & 0 & 1 & 2 & 2 \\ 1 & 1 & 1 & 1 & 2 & 2 \\ 0 & 1 & 2 & 2 & 4 & 4 \\ 0 & 2 & 2 & 2 & 4 & 5 \end{smallmatrix}\right] , \ 
 \left[ \begin{smallmatrix} 0 & 0 & 0 & 0 & 1 & 0 \\ 0 & 1 & 1 & 1 & 3 & 4 \\ 0 & 1 & 2 & 2 & 3 & 4 \\ 0 & 1 & 2 & 2 & 4 & 4 \\ 1 & 3 & 3 & 4 & 7 & 8 \\ 0 & 4 & 4 & 4 & 8 & 9 \end{smallmatrix}\right] , \ 
 \left[ \begin{smallmatrix} 0 & 0 & 0 & 0 & 0 & 1 \\ 0 & 2 & 2 & 2 & 4 & 3 \\ 0 & 2 & 2 & 2 & 4 & 4 \\ 0 & 2 & 2 & 2 & 4 & 5 \\ 0 & 4 & 4 & 4 & 8 & 9 \\ 1 & 3 & 4 & 5 & 9 & 11 \end{smallmatrix} \right]} $$
\item Formal codegrees: $[3, 4, 5, 8, 11, 1320]$,
\item Property: simple, $3$-positive, non-$1$-Frobenius,
\item Categorification: open, non-braided.
\end{itemize}

Observe that it is not $1$-Frobenius because there are basic $\FPdims$ ($9$ and $21$) that do not divide the global $\FPdim$ \(1320 = 2^3  3^1  5^1 11\). Therefore, a categorification of this would serve as a counterexample to the extended version of Kaplansky's sixth conjecture for fusion categories \cite[Question 1]{ENO11}. It is already established that it does not allow a braided categorification, by \cite[Corollary 9.3.5]{EGNO15}. More broadly, any exotic simple integral fusion ring with a rank up to 13 does not permit a braided categorification. If it did, it would be modular according to \cite[Theorem 5.2]{LPR23}, but there is no non-pointed simple integral modular fusion category with a rank up to 13, by \cite{ABPP}.

Our attempt to classify all possible induction matrices produced $2234516$ solutions just for the lower square part (involving $I(\one)$; see~\S\ref{subsub:rel}), available in the \verb|Data/InductionMatrices| directory of~\cite{CodeData}. 
Notably, no non-$1$-Frobenius Drinfeld ring of rank at most~$5$ admits an induction matrix, but there is (at least) one at rank~$6$; see the smallest example in~\cite[\S\ref{subsub:non1FrobR6}(\ref{[1,1,1,3,12,18][0,2,1,3,4,5]})]{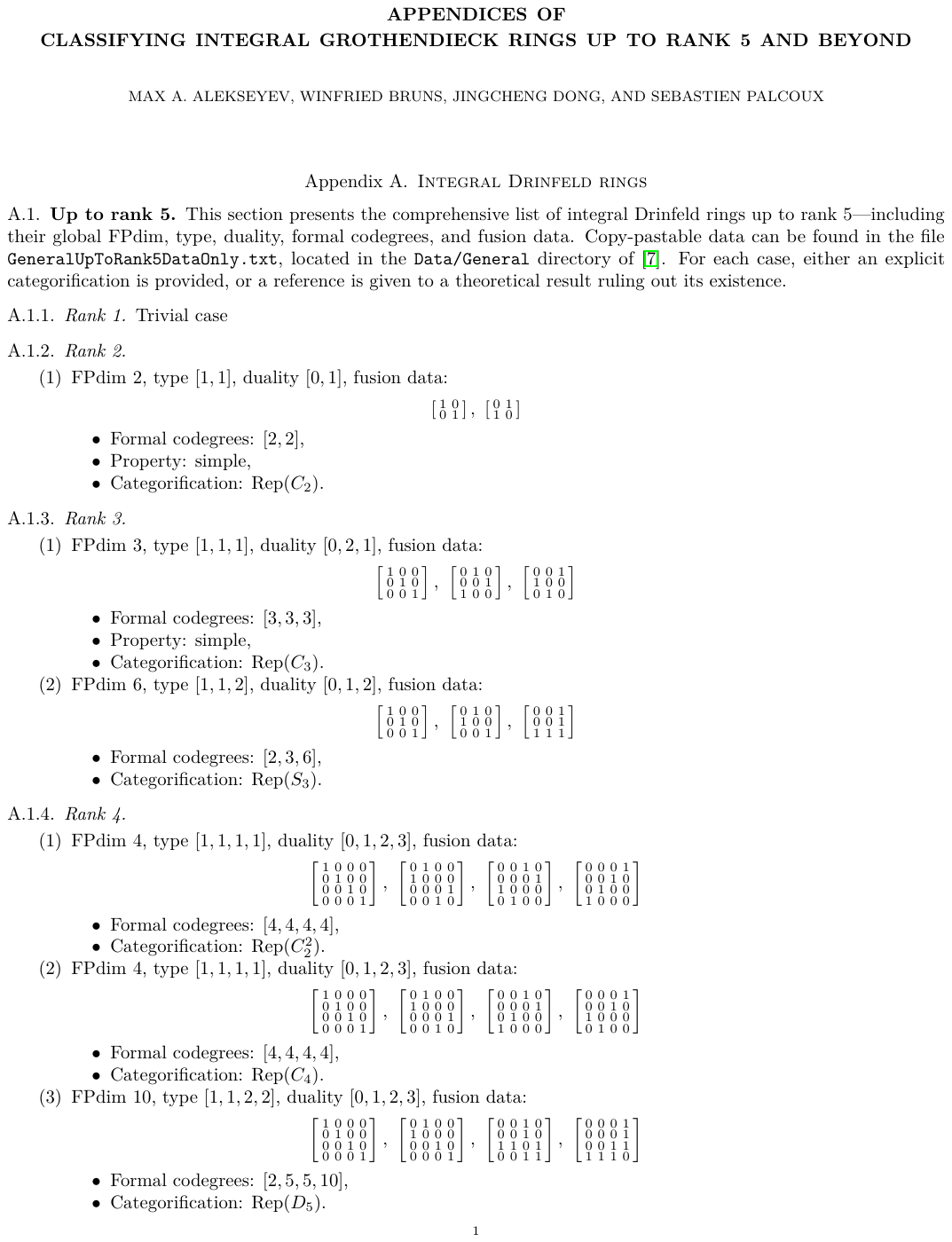}.

%

\begin{theorem} \label{thm:mainexotic2}
The smallest rank for an exotic simple integral $1$-Frobenius Drinfeld ring is $7$.
\end{theorem}

Concerning the exotic cases referenced in Theorem~\ref{thm:mainexotic2}, there are exactly three examples with $\FPdim \le 10^5$ (see~\cite[\S\ref{sub:rank7}]{appendices}), and we expect that no others exist at this rank. The example given below is the only one that is $3$-positive. It represents the smallest example within the interpolated family identified in \cite{LPR23}. Its exclusion from fusion categorification in \cite{LPR2203} necessitated the introduction of the concept of Triangular Prism Equations (TPE). 
\begin{itemize}
\item $\FPdim \ 210$, type $[1,5,5,5,6,7,7]$, duality $[0,1,2,3,4,5,6]$, fusion data:
$$ \normalsize{
\begin{smallmatrix}1&0&0&0&0&0&0 \\ 0&1&0&0&0&0&0 \\ 0&0&1&0&0&0&0 \\ 0&0&0&1&0&0&0 \\ 0&0&0&0&1&0&0 \\ 0&0&0&0&0&1&0 \\ 0&0&0&0&0&0&1\end{smallmatrix} ,\ 
\begin{smallmatrix} 0&1&0&0&0&0&0 \\ 1&1&0&1&0&1&1 \\ 0&0&1&0&1&1&1 \\ 0&1&0&0&1&1&1 \\ 0&0&1&1&1&1&1 \\ 0&1&1&1&1&1&1 \\ 0&1&1&1&1&1&1\end{smallmatrix} ,\ 
\begin{smallmatrix} 0&0&1&0&0&0&0 \\ 0&0&1&0&1&1&1 \\ 1&1&1&0&0&1&1 \\ 0&0&0&1&1&1&1 \\ 0&1&0&1&1&1&1 \\ 0&1&1&1&1&1&1 \\ 0&1&1&1&1&1&1\end{smallmatrix} ,\ 
\begin{smallmatrix} 0&0&0&1&0&0&0 \\ 0&1&0&0&1&1&1 \\ 0&0&0&1&1&1&1 \\ 1&0&1&1&0&1&1 \\ 0&1&1&0&1&1&1 \\ 0&1&1&1&1&1&1 \\ 0&1&1&1&1&1&1\end{smallmatrix} ,\ 
\begin{smallmatrix} 0&0&0&0&1&0&0 \\ 0&0&1&1&1&1&1 \\ 0&1&0&1&1&1&1 \\ 0&1&1&0&1&1&1 \\ 1&1&1&1&1&1&1 \\ 0&1&1&1&1&2&1 \\ 0&1&1&1&1&1&2\end{smallmatrix} ,\ 
\begin{smallmatrix} 0&0&0&0&0&1&0 \\ 0&1&1&1&1&1&1 \\ 0&1&1&1&1&1&1 \\ 0&1&1&1&1&1&1 \\ 0&1&1&1&1&2&1 \\ 1&1&1&1&2&1&2 \\ 0&1&1&1&1&2&2\end{smallmatrix} ,\ 
\begin{smallmatrix} 0&0&0&0&0&0&1 \\ 0&1&1&1&1&1&1 \\ 0&1&1&1&1&1&1 \\ 0&1&1&1&1&1&1 \\ 0&1&1&1&1&1&2 \\ 0&1&1&1&1&2&2 \\ 1&1&1&1&2&2&1 \end{smallmatrix} } $$
\item Formal codegrees: $[5,5,6,7,7,7,210]$,
\item Property: simple, $3$-positive, interpolation of $\ch({\rm PSL}(2,q))$ to $q=6$, see \cite{LPR23},
\item Categorification: excluded in \cite{LPR2203}.
\end{itemize}


Motivated by discussions with Scott Morrison and Pavel Etingof:

\begin{question} \label{qu:F210indIntro}
Does the above fusion ring admit an induction matrix?
\end{question}

A negative answer to Question~\ref{qu:F210indIntro} would offer a somewhat more direct argument for excluding this fusion ring than the method used in~\cite{LPR2203}. However, such an approach lies beyond the reach of both our new \Normaliz{} feature and the techniques developed in~\cite{MW17}. Our attempts to enumerate all possible induction matrices yielded $17843535$ solutions for the lower square alone, underscoring the limitations of this method in the high-rank setting.

The forthcoming paper~\cite{BP24} is devoted to the classification of simple integral fusion rings, motivated by the abundance of open questions it uncovers. As a preview, we highlight some additional exotic simple integral $1$-Frobenius $3$-positive Drinfeld rings discovered in the process. At rank~$8$ with $\FPdim \le 20000$, only one such ring was found (see~\S\ref{sub:1FrobRank8}). It is isotype to the Grothendieck ring of~$\Rep(\mathrm{PSL}(2,11))$, yet it was ruled out as a categorification candidate by the zero-spectrum criterion in~\cite{LPR2203}. At rank~$9$ with $\FPdim \le 10000$, four such rings were identified; all remain open to categorification. The smallest among them (see~\S\ref{sub:1FrobRank9}) is described below.

\begin{itemize}
\item $\FPdim \ 504$, type $[1, 7, \textcolor{blue}{7}, \textcolor{blue}{7}, \textcolor{blue}{7}, 8, 9, 9, 9]$, duality $[0,1,2,3,4,5,6,7,8]$, fusion data:
{\fontsize{7}{8}\selectfont $$ 
\begin{smallmatrix}1&0&0&0&0&0&0&0&0 \\ 0&1&0&0&0&0&0&0&0 \\ 0&0&1&0&0&0&0&0&0 \\ 0&0&0&1&0&0&0&0&0 \\ 0&0&0&0&1&0&0&0&0 \\ 0&0&0&0&0&1&0&0&0 \\ 0&0&0&0&0&0&1&0&0 \\ 0&0&0&0&0&0&0&1&0 \\ 0&0&0&0&0&0&0&0&1\end{smallmatrix} ,   \ 
\begin{smallmatrix}0&1&0&0&0&0&0&0&0 \\ 1&0&1&1&1&0&1&1&1 \\ 0&1&1&0&0&1&1&1&1 \\ 0&1&0&1&0&1&1&1&1 \\ 0&1&0&0&1&1&1&1&1 \\ 0&0&1&1&1&1&1&1&1 \\ 0&1&1&1&1&1&1&1&1 \\ 0&1&1&1&1&1&1&1&1 \\ 0&1&1&1&1&1&1&1&1\end{smallmatrix} ,   \ 
\begin{smallmatrix}0&0&1&0&0&0&0&0&0 \\ 0&1&1&0&0&1&1&1&1 \\ 1&1&\textcolor{blue}{\textbf{0}}&\textcolor{blue}{\textbf{1}}&\textcolor{blue}{\textbf{1}}&0&1&1&1 \\ 0&0&\textcolor{blue}{\textbf{1}}&\textcolor{blue}{\textbf{1}}&\textcolor{blue}{\textbf{0}}&1&1&1&1 \\ 0&0&\textcolor{blue}{\textbf{1}}&\textcolor{blue}{\textbf{0}}&\textcolor{blue}{\textbf{1}}&1&1&1&1 \\ 0&1&0&1&1&1&1&1&1 \\ 0&1&1&1&1&1&1&1&1 \\ 0&1&1&1&1&1&1&1&1 \\ 0&1&1&1&1&1&1&1&1\end{smallmatrix} ,   \ 
\begin{smallmatrix}0&0&0&1&0&0&0&0&0 \\ 0&1&0&1&0&1&1&1&1 \\ 0&0&\textcolor{blue}{\textbf{1}}&\textcolor{blue}{\textbf{1}}&\textcolor{blue}{\textbf{0}}&1&1&1&1 \\ 1&1&\textcolor{blue}{\textbf{1}}&\textcolor{blue}{\textbf{0}}&\textcolor{blue}{\textbf{1}}&0&1&1&1 \\ 0&0&\textcolor{blue}{\textbf{0}}&\textcolor{blue}{\textbf{1}}&\textcolor{blue}{\textbf{1}}&1&1&1&1 \\ 0&1&1&0&1&1&1&1&1 \\ 0&1&1&1&1&1&1&1&1 \\ 0&1&1&1&1&1&1&1&1 \\ 0&1&1&1&1&1&1&1&1\end{smallmatrix} ,   \ 
\begin{smallmatrix}0&0&0&0&1&0&0&0&0 \\ 0&1&0&0&1&1&1&1&1 \\ 0&0&\textcolor{blue}{\textbf{1}}&\textcolor{blue}{\textbf{0}}&\textcolor{blue}{\textbf{1}}&1&1&1&1 \\ 0&0&\textcolor{blue}{\textbf{0}}&\textcolor{blue}{\textbf{1}}&\textcolor{blue}{\textbf{1}}&1&1&1&1 \\ 1&1&\textcolor{blue}{\textbf{1}}&\textcolor{blue}{\textbf{1}}&\textcolor{blue}{\textbf{0}}&0&1&1&1 \\ 0&1&1&1&0&1&1&1&1 \\ 0&1&1&1&1&1&1&1&1 \\ 0&1&1&1&1&1&1&1&1 \\ 0&1&1&1&1&1&1&1&1\end{smallmatrix} ,   \ 
\begin{smallmatrix}0&0&0&0&0&1&0&0&0 \\ 0&0&1&1&1&1&1&1&1 \\ 0&1&0&1&1&1&1&1&1 \\ 0&1&1&0&1&1&1&1&1 \\ 0&1&1&1&0&1&1&1&1 \\ 1&1&1&1&1&1&1&1&1 \\ 0&1&1&1&1&1&2&1&1 \\ 0&1&1&1&1&1&1&2&1 \\ 0&1&1&1&1&1&1&1&2\end{smallmatrix} ,   \ 
\begin{smallmatrix}0&0&0&0&0&0&1&0&0 \\ 0&1&1&1&1&1&1&1&1 \\ 0&1&1&1&1&1&1&1&1 \\ 0&1&1&1&1&1&1&1&1 \\ 0&1&1&1&1&1&1&1&1 \\ 0&1&1&1&1&1&2&1&1 \\ 1&1&1&1&1&2&1&1&2 \\ 0&1&1&1&1&1&1&2&2 \\ 0&1&1&1&1&1&2&2&1\end{smallmatrix} ,   \ 
\begin{smallmatrix}0&0&0&0&0&0&0&1&0 \\ 0&1&1&1&1&1&1&1&1 \\ 0&1&1&1&1&1&1&1&1 \\ 0&1&1&1&1&1&1&1&1 \\ 0&1&1&1&1&1&1&1&1 \\ 0&1&1&1&1&1&1&2&1 \\ 0&1&1&1&1&1&1&2&2 \\ 1&1&1&1&1&2&2&1&1 \\ 0&1&1&1&1&1&2&1&2\end{smallmatrix} ,   \ 
\begin{smallmatrix}0&0&0&0&0&0&0&0&1 \\ 0&1&1&1&1&1&1&1&1 \\ 0&1&1&1&1&1&1&1&1 \\ 0&1&1&1&1&1&1&1&1 \\ 0&1&1&1&1&1&1&1&1 \\ 0&1&1&1&1&1&1&1&2 \\ 0&1&1&1&1&1&2&2&1 \\ 0&1&1&1&1&1&2&1&2 \\ 1&1&1&1&1&2&1&2&1\end{smallmatrix} $$}
\item Formal codegrees: $[7, 7, 7, 8, 9, 9, 9, 9, 504]$,
\item Property: simple, $3$-positive, isotype to $\Rep({\rm PSL}(2,8))$,
\item Categorification: open.
\end{itemize}

This candidate appears to be the most compelling exotic Drinfeld ring currently open to categorification. It represents a subtle variant of the Grothendieck ring of $\Rep(\mathrm{PSL}(2,8))$, differing only in the fusion coefficients involving the last three basic elements of $\FPdim = 7$ (according to the ordering above). To illustrate this variation, it suffices to examine the three corresponding $3 \times 3$ submatrices from each fusion ring.
$$\left[
\begin{smallmatrix} 1&0&1 \\ 0&1&1 \\ 1&1&0 \end{smallmatrix}  ,   \ 
\begin{smallmatrix} 0&1&1 \\ 1&1&0 \\ 1&0&1 \end{smallmatrix}  ,   \ 
\begin{smallmatrix} 1&1&0 \\ 1&0&1 \\ 0&1&1 \end{smallmatrix} 
\right] \to
\left[
\textcolor{blue}{ \begin{smallmatrix} 0&1&1 \\ 1&1&0 \\ 1&0&1 \end{smallmatrix} } ,   \ 
\textcolor{blue}{\begin{smallmatrix} 1&1&0 \\ 1&0&1 \\ 0&1&1 \end{smallmatrix} }  ,   \ 
\textcolor{blue}{\begin{smallmatrix} 1&0&1 \\ 0&1&1 \\ 1&1&0 \end{smallmatrix} }
\right]
$$
Observe that if $x$ is one of the three basic elements mentioned above, then $N_{x,x}^x = 1$ for $\Rep(\mathrm{PSL}(2,8))$, while $N_{x,x}^x = 0$ in the variation. This confirms that the two fusion rings are not isomorphic.
\begin{question}
Can such a slight variation be realized at the categorical level?
\end{question}

The main challenge in the near-future classification of integral fusion categories lies in the exotic simple integral Drinfeld rings. While this paper completes the classification of integral fusion categories up to rank~$5$, progressing to rank~$6$ will inevitably require addressing the exotic simple integral Drinfeld rings of $\FPdim$ $1320$ discussed earlier. A similar situation occurs in the $1$-Frobenius case: what we expect to be the only exotic simple examples at ranks~$7$ and~$8$ were already treated and resolved in~\cite{LPR2203} (at least in the unitary case). Advancing to rank~$9$ will likewise necessitate confronting the exotic simple integral Drinfeld rings of $\FPdim$ ~$504$ mentioned above.



\subsection{Organization} \label{sub:Orga}

\paragraph*{\S\ref{sec:Fu} -- Fusion rings.} After recalling some foundational concepts in~\S\ref{sub:Fu}, this section reviews the notion of formal codegrees in~\S\ref{sub:formal}, presenting two Egyptian fraction decompositions summing to one in the noncommutative setting. In~\S\ref{sub:Drin}, we introduce the notion of Drinfeld rings, motivated by properties of the formal codegrees arising in the Grothendieck ring of a pseudo-unitary fusion category. Finally,~\S\ref{sub:ext} discusses a universal construction for extending an integral fusion ring.

\paragraph*{\S\ref{sec:IndMat} -- Induction matrices.} After a review of the basics in~\S\ref{sub:bas}, based on the framework developed in~\cite[\S9.2]{EGNO15} and~\cite{MW17}, we collect in~\S\ref{sub:sum} the key parameters, variables, and structural relations needed to implement the underlying algebraic system. Further constraints are derived from the ring homomorphism induced by the forgetful functor, as discussed in~\S\ref{sub:ring}. Finally,~\S\ref{sub:NormInd} presents a new feature of \Normaliz{}, enabling a complete classification of possible induction matrices.

\paragraph*{\S\ref{sec:AlgoDetails} -- Algorithmic details.} This section presents the algorithms employed for the classification of Egyptian fractions (Section~\ref{sub:EgyAlgo}) using \SageMath{}, and for the classification of fusion rings (Section~\ref{sub:Normaliz}) and induction matrices (Section~\ref{sub:NormInd}) using \Normaliz{}.

\paragraph*{\S\ref{sec:ExcMod} -- Exclusions via induction matrices.} This section compiles all exclusions derived from induction matrix constraints.

\paragraph*{\S\ref{sec:grpth} -- Group-theoretical models.} After recalling the framework of group-theoretical fusion categories \( \mathcal{C}(G, \omega, H, \psi) \) in~\S\ref{sub:grpthbasic}, we show in~\S\ref{sub:r7d60nc} that the noncommutative Drinfeld ring of type $[1,1,1,3,4,4,4]$ is group-theoretical. We also provide \GAP{} code to verify this automatically for categories of the form \( \mathcal{C}(G, 1, H, 1) \). In~\S\ref{sub:schur}, we use Schur multiplier arguments to reduce certain cases to this form, leading to the exclusion of a fusion category of type $[1,1,1,3,3,21,21]$ in~\S\ref{sub:r7d903}. Finally,~\S\ref{sub:NonIsaacs} exhibits the first example of a non-Isaacs integral fusion category. 

\paragraph*{\S\ref{sec:General} -- Integral Drinfeld rings.} This section provides a summary of the computations used to classify all integral Drinfeld rings of rank at most~$5$ in general, and of higher ranks under additional assumptions. The table below displays the corresponding bounds and the number of Drinfeld rings identified in each case:
\begin{center}
\begin{tabular}{c|c|c|c|c}
\S & \textbf{Rank} & \textbf{Case} & \textbf{Bound on $\FPdim$} & \textbf{Number of Drinfeld rings} \\
\hline
\ref{sub:MainProof} & $\le 5$ & All & All & $29$ \\
\ref{subsub:1FrobRank6} & $6$ & $1$-Frobenius & All & $58$ \\
\ref{subsub:N1FrobRank6} & $6$ & Non-$1$-Frobenius & $\le 200000$ & $88$ \\
\ref{subsub:1FrobRank7} & $7$ & $1$-Frobenius & $\le 100000$ & $241$ \\
\ref{subsub:N1FrobRank7} & $7$ & Non-$1$-Frobenius & $\le 5000$ & $113$ \\
\ref{sub:1FrobRank8} & $8$ & $1$-Frobenius & $\le 20000$ & $750$ \\
\ref{sub:1FrobRank9} & $9$ & $1$-Frobenius & $\le 2000$ & $1292$
\end{tabular}
\end{center}

\paragraph*{\S\ref{sec:Odd} -- Odd-dimensional case.} Results from~\cite{NS07} on Frobenius--Schur indicators impose strong constraints on odd-dimensional integral Grothendieck rings, affecting their type, duality, and formal codegrees (see~\S\ref{sub:MNSDRings}). These constraints motivate the introduction of MNSD Drinfeld rings (Definition~\ref{def:MNSDring}). The table below summarizes the classification obtained in this context, leading to a complete classification of all odd-dimensional integral Grothendieck rings up to rank~$7$ (Theorem~\ref{thm:mainodd}). The first example known to us of an odd-dimensional integral fusion category that is not Grothendieck equivalent to any Tannakian category occurs at rank~$27$; see~\S\ref{sub:NonTann}.
\begin{center}
\begin{tabular}{c|c|c|c|c}
\S & \textbf{Rank} & \textbf{Case} & \textbf{Bound on $\FPdim$} & \textbf{Number of Drinfeld rings} \\
\hline
\ref{sub:MNSDGeneralUpToR7} & $\le 7$ & All MNSD & All & $8$ \\
\ref{subsub:1FrobMNSDR9} & $9$ & $1$-Frobenius & All & $10$ \\
\ref{subsub:N1FrobMNSDR9} & $9$ & Non-perfect non-$1$-Frobenius & All & $2$ \\
\ref{subsub:N1FrobMNSDR9} & $9$ & Perfect non-$1$-Frobenius & $< 389865$ & $0$ \\
\ref{sub:MNSD1FrobR11} & $11$ & Non-perfect $1$-Frobenius & $\le 10^9$ & $24$
\end{tabular}
\end{center}
Our classification strategy begins with MNSD Egyptian fractions, in the spirit of~\cite{ABPP}, with the key difference that the fractions sum to~$1$ and do not require squared denominators.

\paragraph*{\S\ref{sec:NC} -- Noncommutative case.} In~\S\ref{FuRingsNC}, we present several constraints on isomorphism classes of noncommutative complexified fusion rings, using Galois-theoretic arguments. In particular, we show that rank~$6$ is the minimal rank for a noncommutative fusion ring.
In~\S\ref{DrRingsNC}, we study Drinfeld rings and show—using Egyptian fraction techniques—that the group ring \(\mathbb{Z}S_3\) is the unique noncommutative integral Drinfeld ring of rank~6, and thus the only integral Grothendieck ring of that rank. We then extend the classification to all integral Grothendieck rings of rank up to~7, all noncommutative integral Drinfeld rings of rank up to~8, and finally to noncommutative integral \(1\)-Frobenius Drinfeld rings of rank~9 with \(\FPdim \le 10000\).
In~\S\ref{GrRingsNC}, leveraging results involving Frobenius--Schur indicators, we prove that any noncommutative, odd-dimensional integral Grothendieck ring must have rank at least~$21$, with $C_7 \rtimes C_3$ as the unique example at this rank. We conclude with a discussion of the rank~$23$ case.

\begin{center}
\begin{tabular}{c|c|c|c|c}
\S & \textbf{Rank} & \textbf{Case} & \textbf{Bound on $\FPdim$} & \textbf{Number of Drinfeld rings} \\
\hline
\ref{subsub:NCR7} & $\le 7$ & NC Grothendieck & All & $3$ \\
\ref{subsub:NCR8} & $\le 8$ & NC Drinfeld & All & $29$ \\
\ref{subsub:NCR9} & $9$ & $1$-Frobenius NC & $\le 10000$ & $83$ \\
\ref{GrRingsNC} & $\le 21$ & Grothendieck + MNSD + NC & All & $1$
\end{tabular}
\end{center}

\paragraph*{\cite[\ref{sec:GeneralA}, \ref{sec:OddA}, \ref{sec:NCA}]{appendices}} They collect the complete fusion data in the general, MNSD, and noncommutative settings, respectively. They also include references for the exclusions or the models relevant to the categorification. In this sense, they constitute an integral part of the proofs of the main theorems of the paper.

\section{Fusion rings} \label{sec:Fu}
After recalling some basics in~\S\ref{sub:Fu}, this section reviews the notion of formal codegrees in~\S\ref{sub:formal}, presenting two main Egyptian fraction decompositions summing to one in the noncommutative setting. We then introduce the notion of Drinfeld rings in~\S\ref{sub:Drin}, motivated by properties involving the formal codegrees of the Grothendieck ring of a pseudo-unitary fusion category. Finally, we discuss a universal way for extending an integral fusion ring in~\S\ref{sub:ext}.

\subsection{Basics} \label{sub:Fu}

In this subsection, we review the concept of fusion data, along with the essential results. For further details, we refer the reader to \cite{EGNO15}. The concept of fusion data expands upon the idea of a finite group.

\begin{definition} \label{def:fu}
\emph{Fusion data} consist of a finite set $\{1,2,...,r\}$ with an involution $i \mapsto i^*$, and nonnegative integers $N_{i,j}^k$ satisfying the following conditions for all $i,j,k,t$:
\begin{itemize}
\item (Associativity) $\sum_s N_{i,j}^s N_{s,k}^t = \sum_s N_{j,k}^s N_{i,s}^t$,
\item (Unit) $N_{1,i}^j = N_{i,1}^j = \delta_{i,j}$,
\item (Dual) $N_{i^*,j}^{1} = N_{j,i^*}^{1} = \delta_{i,j}$,
\item (Anti-involution) $N_{i,j}^{k} = N_{j^*,i^*}^{k^*}$.
\end{itemize}
Note that $1^* = 1$. We may represent the fusion data simply as $(N_{i,j}^k)$.
\end{definition}

\begin{proposition}[Frobenius Reciprocity] \label{prop:FrobRec}
For all $i,j,k$, $N_{i,j}^k = N_{k,j^*}^{i} = N_{k^*,i}^{j^*} =  N_{j^*,i^*}^{k^*} = N_{j,k^*}^{i^*} = N_{i^*,k}^j$.
\end{proposition}
\begin{proof}
Starting with (Associativity) and setting $t=1$, we have $\sum_s N_{i,j}^s N_{s,k}^1 = \sum_s N_{j,k}^s N_{i,s}^1$. Applying (Dual), we get $\sum_s N_{i,j}^s \delta_{s,k^*} = \sum_s N_{j,k}^s \delta_{s,i^*}$. Consequently, $N_{i,j}^{k^*} = N_{j,k}^{i^*}$. Substituting $k^*$ with $k$, we obtain $N_{i,j}^{k} = N_{j,k^*}^{i^*}$, which equals $N_{k,j^*}^{i}$ by (Anti-involution). The proposition follows by iterating the equality $N_{i,j}^k = N_{k,j^*}^{i}$ and (Anti-involution). 
\end{proof}

\begin{remark}[Group case] \label{rem:grp}
Let $G = \{g_1, \dots, g_{r}\}$ be a finite group where $g_1=e$ is the neutral element, and define the involution by the inverse, i.e. $g_{i^*} = g_i^{-1}$. Then the nonnegative integers $N_{i,j}^k := \delta_{g_ig_j, g_k}$ define a fusion data, because in this case, the three first axioms above are exactly the axioms of a group, whereas the fourth one corresponds to the equality $(gh)^{-1} = h^{-1} g^{-1}$.
\end{remark}

\begin{remark}
We can construct data that satisfy the first three axioms of Definition \ref{def:fu} but not the fourth, proving it is not superfluous. However, (Unit) is redundant when combined with the other axioms, as it is not utilized in the proof of Proposition \ref{prop:FrobRec}. Taken together, (Dual) and (Frobenius Reciprocity) trivially imply (Unit).
\end{remark}

A \emph{fusion ring} $\mathcal{R}$ is a free $\mathbb{Z}$-module equipped with a finite basis $\mathcal{B}=\{b_1, \dots, b_{r}\}$ and a fusion product defined by $$ b_i  b_j = \sum_k N_{i,j}^k b_k, $$ where $(N_{i,j}^k)$ constitutes fusion data, and a $*$-structure given by $b_i^* := b_{i^*}$. The four axioms for fusion data translate to the following for all $i,j,k$:
\begin{itemize}
\item  $(b_i  b_j)  b_k = b_i  (b_j  b_k)$, 
\item  $b_1  b_i = b_i  b_1 = b_i$,
\item  $\tau(b_i  b_j^*) = \delta_{i,j}$, 
\item  $(b_i  b_j)^* = b_j^*  b_i^*$,
\end{itemize}
where $\tau(x)$ is the coefficient of $b_1$ in the decomposition of $x \in \mathcal{R}$. Consequently, $\mathcal{R}_{\mathbb{C}} := \mathcal{R} \otimes_{\mathbb{Z}} \mathbb{C}$ becomes a finite-dimensional unital $*$-algebra, with $\tau$ extending linearly to a trace (i.e., $\tau(xy) = \tau(yx)$) and an inner product defined by $\langle x,y \rangle := \tau(x y^*)$. Here, $\langle x,b_i \rangle$ is the coefficient of $b_i$ in the decomposition of $x$.

\begin{theorem}[Frobenius-Perron Theorem \cite{EGNO15}] \label{thm:FrobPer}
Given a fusion ring $\mathcal{R}$ with basis $\mathcal{B}$ and the corresponding finite-dimensional unital $*$-algebra $\mathcal{R}_{\mathbb{C}}$, there exists a unique $*$-homomorphism $d:\mathcal{R}_{\mathbb{C}} \to \mathbb{C}$ such that $d(\mathcal{B}) \subset \mathbb{R}_{>0}$.
\end{theorem}

The value $d(b_i)$, known as the \emph{Frobenius-Perron dimension} of $b_i$, is denoted as $\FPdim(b_i)$ or simply $d_i$. This is referred to as a \emph{basic $\FPdim$}. Here is a list of basic invariants of a fusion ring $\mathcal{R}$:
\begin{itemize}
\item \emph{Rank}: the integer $r$,
\item \emph{Global $\FPdim$}: the sum $\sum_i d_i^2$,
\item \emph{Type}: the list $[d_1, d_2, \dots, d_{r}]$, with $d_1 \le d_2 \le \dots \le d_r$,
\item \emph{Duality}: the list \([i^* - 1 \mid i=1,\dots,r]\)\footnote{The subtraction by \(1\) accounts for Python's zero-based indexing convention.} representing the involution,\footnote{It is an invariant up to permutations of the basic elements having the same \(\FPdim\).}
\item \emph{Multiplicity}: the maximum value among $N_{i,j}^k$.
\end{itemize}

The Frobenius--Perron theorem establishes a fundamental identity connecting the type with the fusion rules:
\begin{itemize}
\item (DimEq) \qquad \(d_i d_j = \sum_k N_{i,j}^k d_k, \  i,j = 1,\dots,r.\)
\end{itemize}
These relations are precisely what enable the assignment \(b_i \mapsto d_i\), for \(i=1,\dots,r\), to extend naturally to a ring homomorphism \(\mathcal{R}_{\mathbb{C}} \to \mathbb{C}\).
A fusion ring $\mathcal{R}$ is described as:
\begin{itemize}
\item \emph{$s$-Frobenius} if $\frac{\FPdim(\mathcal{R})^s}{\FPdim(b_i)}$ is an algebraic integer, for all $i$,  
\item \emph{integral} if the number $\FPdim(b_i)$ is an integer, for all $i$,
\item \emph{pointed} if $\FPdim(b_i)=1$, for all $i$,
\item \emph{commutative} if $b_i b_j = b_j  b_i$, for all $i,j$, i.e. $N_{i,j}^k = N_{j,i}^k$,  
\item \emph{simple} if for all $\mathcal{B}' \subset \mathcal{B}$ generating a proper fusion subring then $\mathcal{B}' = \{ b_1 \}$,
\item \emph{perfect} if $d_i=1$ if and only if $i=1$ (i.e. $d_2 > 1$).
\end{itemize}

A fusion ring is pointed if and only if it is a group ring $\mathbb{Z}G$ with basis $\mathcal{B} = G$ a finite group (for the fusion product), its fusion data is the one described in Remark \ref{rem:grp}. Here is an other way to make a fusion ring from a finite group: the character ring $ch(G)$ with basis the set of irreducible characters. As a fusion ring, the group ring $\mathbb{Z}G$ is pointed, but commutative if and only if $G$ is commutative, simple if and only if $G$ is cyclic of prime order and perfect if and only if $G$ is trivial; whereas the character ring $ch(G)$ is commutative,  but pointed if and only if $G$ is abelian, simple if and only if $G$ is simple, and perfect if and only if $G$ is perfect. Both are $1$-Frobenius integral with $\FPdim = |G|$. 

\begin{theorem}[\cite{HLPW24}] \label{thm:n-positive}
The Grothendieck ring of a unitary fusion category satisfies the \emph{primary $n$-criterion}, that is,
\[
\sum_i \|M_i\|^{2-n} M_i^{\otimes n} \geq 0, \quad \text{for all } n \ge 1,
\]
where $M_i$ is the fusion matrix $(N_{i,j}^k)_{k,j}$, $\otimes$ denotes the Kronecker product, and $\|\cdot\|$ is the matrix $\ell^2$-norm.
\end{theorem}
A fusion ring is said to be \emph{$n$-positive} if it satisfies the primary $n$-criterion. In the commutative case, the Schur product criterion from~\cite{LPW21} coincides with the primary $3$-criterion.

\subsection{Formal codegrees} \label{sub:formal} 

Let $\mathcal{R}$ be a fusion ring with basis $(b_i)_{i \in I}$. From \cite[Proposition 3.1.8]{EGNO15}, the element  
\[
Z := \sum_{i \in I} b_i b_{i^*}
\]
is central in $\mathcal{R}$. The complexified algebra, $\mathcal{R}_{\mathbb{C}} := \mathcal{R} \otimes_{\mathbb{Z}} \mathbb{C}$, is a finite-dimensional unital $ * $-algebra. Let ${\rm Irr}(\mathcal{R}_{\mathbb{C}})$ denote the set of irreducible complex representations of $\mathcal{R}_{\mathbb{C}}$, up to equivalence. Then:  
\[
\mathcal{R}_{\mathbb{C}} \simeq \bigoplus_{V \in {\rm Irr}(\mathcal{R}_{\mathbb{C}})} {\rm End}_{\mathbb{C}}(V).
\]
The elements $ V \in {\rm Irr}(\mathcal{R}_{\mathbb{C}}) $ are in one-to-one correspondence with the minimal central projections $ p_V $ in $ \mathcal{R}_{\mathbb{C}} $, such that
\[
p_V \mathcal{R}_{\mathbb{C}} p_V \simeq {\rm End}_{\mathbb{C}}(V) \simeq M_{n_V}(\mathbb{C}),
\]
where $M_{n_V}(\mathbb{C})$ is the algebra of $n_V \times n_V$ complex matrices and $n_V = \dim(V)$.  

Since $Z$ and $p_V$ are central in $\mathcal{R}_{\mathbb{C}}$, it follows that $p_V Z p_V$ is central in $p_V \mathcal{R}_{\mathbb{C}} p_V$. But $M_{n_V}(\mathbb{C})$ has a trivial center, meaning $p_V Z p_V$ must be a scalar multiple of $p_V$:  
\[
p_V Z p_V = \alpha_V p_V.
\]
Following \cite[Lemma 2.6]{Os09}, the scalar $f_V:=\alpha_V/n_V \in \mathbb{C}$ defines the \emph{formal codegree} of $V$.
%
%
In particular, the eigenvalues of the left multiplication matrix $L_Z$ of $Z$ are $\alpha_V=n_V f_V$, with multiplicity $n_V^2 = \dim_{\mathbb{C}}(M_{n_V}(\mathbb{C}))$. The following results from \cite{Os09,Os15} hold:
\begin{itemize}
    \item Both $(f_V)$ and $(\FPdim(b_i))$ are algebraic integers.
    \item We have two (algebraic) Egyptian fractions that sums to $1$:
    \[
    {\rm Tr}(L_Z^{-1}) = \sum_{V \in {\rm Irr}(\mathcal{R}_{\mathbb{C}})} \sum_{j=1}^{n_V^2} \frac{1}{n_Vf_V}  =  \sum_{V \in {\rm Irr}(\mathcal{R}_{\mathbb{C}})} \sum_{j=1}^{n_V} \frac{1}{f_V}  = \sum_{V \in {\rm Irr}(\mathcal{R}_{\mathbb{C}})} \frac{n_V}{f_V} = 1.
    \]
\end{itemize}
If $\mathcal{R}$ is the Grothendieck ring of a fusion category $\mathcal{C}$ over $\mathbb{C}$, then:
\begin{itemize}
    \item The values $(\dim(\mathcal{C})/f_V)$ are algebraic integers.
    \item The numbers $(\FPdim(b_i))$ are cyclotomic integers. 
\end{itemize}
If $\mathcal{R}$ is the Grothendieck ring of a spherical fusion category over $\mathbb{C}$, then:
\begin{itemize}
    \item The formal codegrees $(f_V)$ are cyclotomic integers.
    \item The cyclotomic conductor of $(f_V)$ divides that of $(\dim(b_i))$.
\end{itemize}
In particular, if $\mathcal{R}$ comes from an integral fusion category over $\mathbb{C}$, then:
\begin{itemize}
    \item Both $(f_V)$ and $(\FPdim(b_i))$ are rational integers.
    \item Each $f_V$ divides $\FPdim(\mathcal{R})$.
\end{itemize}
If $\mathcal{R}$ is commutative, each $V$ is one-dimensional. Thus, the representations can be indexed by $I$, allowing us to define $ f_i := f_{V_i} $. Then:
\begin{itemize}
    \item The (algebraic) Egyptian fraction simplifies to 
    \[
    \sum_{i \in I} \frac{1}{f_i} = 1
    \]
    \item The eigenvalues of $L_Z$ are $(f_i)$, with $n_{V_i} = 1$,
\end{itemize}
    
If \(\mathcal{R}\) is the commutative Grothendieck ring of a fusion category \(\mathcal{C}\) over~\(\mathbb{C}\), then the eigenvalues of \(L_Z\) divide \(\dim(\mathcal{C})\) as algebraic integers. This divisibility, however, does not necessarily hold in the noncommutative case. As shown in \cite[\ref{sec:NCA}]{appendices}, there is no such integral Drinfeld ring up to rank~8, but we found several examples at rank~9. 
For instance, one example has \(\FPdim = 24\), type \([1, 1, 1, 1, 2, 2, 2, 2, 2]\), duality \([0, 1, 2, 3, 4, 5, 6, 8, 7]\), formal codegrees \([4, 4, 8_2, 8, 12, 24]\), and the fusion data given below. 
The notation \(8_2\), explained in \cite[\ref{sec:NCA}]{appendices}, means that there exists \(V\) with \((n_V, f_V) = (2, 8)\), but \(2 \times 8\) does not divide \(24\).
{\fontsize{7}{8}\selectfont $$ 
\begin{smallmatrix}1&0&0&0&0&0&0&0&0 \\ 0&1&0&0&0&0&0&0&0 \\ 0&0&1&0&0&0&0&0&0 \\ 0&0&0&1&0&0&0&0&0 \\ 0&0&0&0&1&0&0&0&0 \\ 0&0&0&0&0&1&0&0&0 \\ 0&0&0&0&0&0&1&0&0 \\ 0&0&0&0&0&0&0&1&0 \\ 0&0&0&0&0&0&0&0&1\end{smallmatrix} ,   \ 
\begin{smallmatrix}0&1&0&0&0&0&0&0&0 \\ 1&0&0&0&0&0&0&0&0 \\ 0&0&0&1&0&0&0&0&0 \\ 0&0&1&0&0&0&0&0&0 \\ 0&0&0&0&0&0&0&0&1 \\ 0&0&0&0&0&0&0&1&0 \\ 0&0&0&0&0&0&1&0&0 \\ 0&0&0&0&0&1&0&0&0 \\ 0&0&0&0&1&0&0&0&0\end{smallmatrix} ,   \ 
\begin{smallmatrix}0&0&1&0&0&0&0&0&0 \\ 0&0&0&1&0&0&0&0&0 \\ 1&0&0&0&0&0&0&0&0 \\ 0&1&0&0&0&0&0&0&0 \\ 0&0&0&0&0&0&0&0&1 \\ 0&0&0&0&0&0&0&1&0 \\ 0&0&0&0&0&0&1&0&0 \\ 0&0&0&0&0&1&0&0&0 \\ 0&0&0&0&1&0&0&0&0\end{smallmatrix} ,   \ 
\begin{smallmatrix}0&0&0&1&0&0&0&0&0 \\ 0&0&1&0&0&0&0&0&0 \\ 0&1&0&0&0&0&0&0&0 \\ 1&0&0&0&0&0&0&0&0 \\ 0&0&0&0&1&0&0&0&0 \\ 0&0&0&0&0&1&0&0&0 \\ 0&0&0&0&0&0&1&0&0 \\ 0&0&0&0&0&0&0&1&0 \\ 0&0&0&0&0&0&0&0&1\end{smallmatrix} ,   \ 
\begin{smallmatrix}0&0&0&0&1&0&0&0&0 \\ 0&0&0&0&0&0&0&1&0 \\ 0&0&0&0&0&0&0&1&0 \\ 0&0&0&0&1&0&0&0&0 \\ 1&0&0&1&1&0&0&0&0 \\ 0&0&0&0&0&0&1&0&1 \\ 0&0&0&0&0&1&0&0&1 \\ 0&1&1&0&0&0&0&1&0 \\ 0&0&0&0&0&1&1&0&0\end{smallmatrix} ,   \ 
\begin{smallmatrix}0&0&0&0&0&1&0&0&0 \\ 0&0&0&0&0&0&0&0&1 \\ 0&0&0&0&0&0&0&0&1 \\ 0&0&0&0&0&1&0&0&0 \\ 0&0&0&0&0&0&1&1&0 \\ 1&0&0&1&0&1&0&0&0 \\ 0&0&0&0&1&0&0&1&0 \\ 0&0&0&0&1&0&1&0&0 \\ 0&1&1&0&0&0&0&0&1\end{smallmatrix} ,   \ 
\begin{smallmatrix}0&0&0&0&0&0&1&0&0 \\ 0&0&0&0&0&0&1&0&0 \\ 0&0&0&0&0&0&1&0&0 \\ 0&0&0&0&0&0&1&0&0 \\ 0&0&0&0&0&1&0&1&0 \\ 0&0&0&0&1&0&0&0&1 \\ 1&1&1&1&0&0&0&0&0 \\ 0&0&0&0&1&0&0&0&1 \\ 0&0&0&0&0&1&0&1&0\end{smallmatrix} ,   \ 
\begin{smallmatrix}0&0&0&0&0&0&0&1&0 \\ 0&0&0&0&1&0&0&0&0 \\ 0&0&0&0&1&0&0&0&0 \\ 0&0&0&0&0&0&0&1&0 \\ 0&0&0&0&0&1&1&0&0 \\ 0&1&1&0&0&0&0&1&0 \\ 0&0&0&0&0&1&0&0&1 \\ 0&0&0&0&0&0&1&0&1 \\ 1&0&0&1&1&0&0&0&0\end{smallmatrix} ,   \ 
\begin{smallmatrix}0&0&0&0&0&0&0&0&1 \\ 0&0&0&0&0&1&0&0&0 \\ 0&0&0&0&0&1&0&0&0 \\ 0&0&0&0&0&0&0&0&1 \\ 0&1&1&0&0&0&0&0&1 \\ 0&0&0&0&1&0&1&0&0 \\ 0&0&0&0&1&0&0&1&0 \\ 1&0&0&1&0&1&0&0&0 \\ 0&0&0&0&0&0&1&1&0\end{smallmatrix} $$}



Let $ f_1 $ denote the formal codegree corresponding to the linear character $\FPdim$, then $ f_1 = \FPdim(\mathcal{R}) $. In the pseudo-unitary case—where $\dim(\mathcal{C}) = \FPdim(\mathcal{C})$—every $ f_V $ divides $ f_1 $, and moreover, by \cite[Theorem 2.21]{Os15}, $$ {\rm Tr}(L_Z^{-2}) = \sum_{V \in {\rm Irr}(\mathcal{R}_{\mathbb{C}})} \frac{1}{f_V^2} \le \frac{1}{2}(1+\frac{1}{f_1}).$$

\begin{remark} \label{rk:NCcase}
The noncommutative case gives two Egyptian fractions that sum to $1$, as shown above. However, the second is far more convenient because $ f_V $ divides $ \dim(\mathcal{C}) $. Thus, in the integral (so pseudo-unitary) case, we only need to consider shorter Egyptian fractions that satisfy this divisibility condition—that is, $ f_V $ divides $ f_1 $ for all $ V $. This makes the problem significantly more manageable from a combinatorial perspective (see \cite{A002966} and \cite{A374582} for comparison).
\end{remark}

\subsection{Drinfeld rings} \label{sub:Drin}

The result below follows directly from~\cite[Corollary 8.54]{ENO05} and~\cite[Corollaries 2.14 and 2.15]{Os15}:

\begin{proposition} \label{prop:DrinList}
Let $\mathcal{C}$ be a pseudo-unitary fusion category over~$\mathbb{C}$. Then its Grothendieck ring satisfies the following:
\begin{itemize}
\item The basic $\FPdims$ are cyclotomic integers;
\item The formal codegrees are cyclotomic integers;
\item The conductor of the formal codegrees divides that of the basic $\FPdims$;
\item The ratio of the global $\FPdim$ to each formal codegree is an algebraic integer.
\end{itemize}
\end{proposition}

The divisibility condition in \cite{Os15} is established using the Drinfeld center. A fusion ring that satisfies all the conditions listed in Proposition \ref{prop:DrinList} will be referred to as a \emph{Drinfeld ring}. Since integral fusion categories are pseudo-unitary by \cite[Proposition 9.6.5]{EGNO15}, their Grothendieck rings are always Drinfeld. For an integral fusion ring, the Drinfeld condition simply requires that the formal codegrees be integers dividing the global $\FPdim$.

Building on \cite{ENO11}, \cite[Proposition 5.5]{ABPP} proves that there is no nontrivial perfect integral fusion category whose $\FPdim$ is of the form \( p^a q^b \). However, this result does not extend to fusion rings: several counterexamples are provided in \cite{BP24}, including the following simple integral fusion ring with \emph{prime} $\FPdim$: 
\begin{itemize}
\item $\FPdim \ 532159$, type $[1, 211, 409, 566]$, duality $[0,1,2,3]$, fusion data: 
$$ 
\left[\begin{matrix} 1&0&0&0 \\ 0&1&0&0 \\ 0&0&1&0 \\ 0&0&0&1\end{matrix} \right],  \ \left[ \begin{matrix}0&1&0&0 \\ 1&84&24&30 \\ 0&24&63&98 \\ 0&30&98&129\end{matrix} \right],  \ \left[ \begin{matrix}0&0&1&0 \\ 0&24&63&98 \\ 1&63&299&56 \\ 0&98&56&332\end{matrix} \right],  \ \left[ \begin{matrix}0&0&0&1 \\ 0&30&98&129 \\ 0&98&56&332 \\ 1&129&332&278\end{matrix}\right]
$$
\end{itemize}

None of such perfect fusion rings discovered so far are Drinfeld or $1$-Frobenius. This naturally raises the following question:

\begin{question} \label{qu:BadPerfect}
Is there a perfect integral fusion ring with $\FPdim$ \( p^a q^b \) that is either Drinfeld or \( 1 \)-Frobenius?
\end{question}

\subsection{Integral extension} \label{sub:ext}

\begin{proposition} \label{prop:ext1}
Let $\mathcal{R}$ be an integral fusion ring of rank $n$ with basis $\{b_1, \dots, b_{n}\}$ and type $[d_1, \dots, d_{n}]$, where $d_i = \FPdim(b_i)$. Then there exists an integral fusion ring $\tilde{\mathcal{R}}$ of rank $n+1$, extending $\mathcal{R}$, with an additional basis element $\rho$ satisfying $\FPdim(\rho) = \FPdim(\mathcal{R})$, and the fusion rules:  
\[
\rho b_i = b_i \rho = d_i \rho, \quad \text{and} \quad \rho^2 = \sum_i d_i b_i + (\FPdim(\mathcal{R}) - 1)\rho.
\]
\end{proposition}

\begin{proof}
It is straightforward to verify that $\tilde{\mathcal{R}}$ satisfies all the axioms in Definition \ref{def:fu}. Since $\rho$ is central, associativity reduces to checking that for all $i, j$:
\[
(\rho b_i) b_j = \rho (b_i b_j), \quad (\rho \rho) b_i = \rho (\rho b_i), \quad \text{and} \quad (\rho \rho) \rho = \rho (\rho \rho).
\]
To prove that $\FPdim(\rho) = \FPdim(\mathcal{R})$, it suffices to solve the equation  
\[
X^2 - (\FPdim(\mathcal{R}) - 1)X - \FPdim(\mathcal{R}) = 0. \qedhere
\]
\end{proof}

\begin{remark}
The proof of Proposition~\ref{prop:ext1} extends to the case where $N_{\rho,\rho}^\rho = \left(\FPdim(\mathcal{R})/n - n\right)$ is a non-negative integer. In that case, $\FPdim(\rho) = \FPdim(\mathcal{R})/n$. However, the resulting extension may fail to be Drinfeld, even if $\mathcal{R}$ is. For instance, when $R = \ch(A_5)$, the extension is Drinfeld for $n = 1,2$ (see \cite[\S\ref{subsub:1FrobR6}(\ref{[1,3,3,4,5,30]}) and (\ref{[1,3,3,4,5,60]})]{appendices}), but not for $n = 3$. The extension for $n = 1$ of an integral Drinfeld ring may always be Drinfeld.
\end{remark}

By iterating Proposition \ref{prop:ext1} starting from the trivial fusion ring, we obtain a sequence of fusion rings with types  
\[
[1], [1,1], [1,1,2], [1,1,2,6], \dots,
\]
related to the Fibonacci-like sequence defined by  
\[
u_{n+1} = u_n + u_n^2, \quad u_0 = 1,
\]
whose first terms are $1, 2, 6, 42, 1806, \dots$ (see \cite{A007018}). The fusion ring of type $[1,1,2,6]$ does not admit a complex categorification, as shown in Lemma \ref{lem:1126}. More generally, by iterating Proposition \ref{prop:ext1} starting from any integral fusion ring $\mathcal{R} = \mathcal{R}_0$, we can construct a sequence of fusion rings $\mathcal{R}_n$.

\begin{conjecture} \label{conj:ext}
For every integral fusion ring $\mathcal{R} = \mathcal{R}_0$, there exists an integer $n$ such that $\mathcal{R}_n$ does not admit a complex categorification.
\end{conjecture}

The smallest such $n$ defines an invariant of the integral fusion ring $\mathcal{R}$, denoted $ n(\mathcal{R}) $. In particular,  
\[
n(\mathcal{R}) = 0 \quad \text{if and only if} \quad \mathcal{R} \text{ does not admit a complex categorification}.
\]
Moreover, we note that $ n(1) = 3 $, where $ 1 $ denotes the trivial fusion ring.

\begin{question} \label{qu:ext}
Is there a fusion ring $\mathcal{R}$ with $ n(\mathcal{R}) > 3 $?
\end{question}

\section{Induction Matrices} \label{sec:IndMat}

After recalling some basics in~\S\ref{sub:bas} concerning the notion of induction matrices—as discussed, for instance, in~\cite[\S 9.2]{EGNO15} and~\cite{MW17}—we gather in~\S\ref{sub:sum} the relevant parameters, variables, and relations with the aim of explicitly implementing the underlying system. Additional relations arise from the ring homomorphism induced by the forgetful functor, as explained in~\S\ref{sub:ring}. Finally, in~\S\ref{sub:NormInd}, we present our new \Normaliz{} feature, which enables a complete classification of all possible induction matrices.

\subsection{Basics} \label{sub:bas}

Let ${\rm R}$ be a fusion ring with fusion data $(N_{i,j}^k)$ and basis $\{a_1, \ldots, a_{r}\}$, where $a_1$ is the unit. For simplicity, we assume that ${\rm R}$ is integral. Assume that ${\rm R}$ admits a categorification into a fusion category $\mathcal{C}$ over the complex field. Then the Drinfeld center $\mathcal{Z}(\mathcal{C})$ of $\mathcal{C}$ is an integral modular fusion category. Let ${\rm ZR}$ be the Grothendieck ring of $\mathcal{Z}(\mathcal{C})$, which is an integral commutative half-Frobenius fusion ring. Let $\{b_1, \ldots, b_{n}\}$ be the basis of ${\rm ZR}$.

\begin{remark} Multiple ${\rm ZR}$ are possible because a fusion ring can have several non-equivalent categorifications $\mathcal{C}$. For instance, in the pointed case, fusion rings are represented by group rings $\mathbb{Z}G$. However, for any given finite group $G$, there are multiple categorifications as the category of $G$-graded vector spaces $\text{Vec}(G,\omega)$ twisted by a 3-cocycle $\omega$. When considering two non-equivalent 3-cocycles $\omega_1$ and $\omega_2$, the Grothendieck rings of $\mathcal{Z}(\text{Vec}(G,\omega_1))$ and $\mathcal{Z}(\text{Vec}(G,\omega_2))$ can be non-isomorphic fusion rings. For further details, refer to \cite{GrMo}.
\end{remark}

Let $d_i := \FPdim(a_i)$ and $m_i := \FPdim(b_i)$. Then, $\FPdim({\rm R}) := \sum_i d_i^2$ and $\FPdim({\rm ZR}) := \sum_i m_i^2$. Note that \cite[Theorem 7.16.6]{EGNO15} states that $\FPdim({\rm ZR}) = \FPdim({\rm R})^2$. However, by half-Frobenius property, $m_i^2$ divides $\FPdim({\rm ZR})$, so $m_i$ divides $\FPdim({\rm R})$. There is a ring morphism $F: {\rm ZR} \to {\rm R}$ preserving $\FPdim$, induced by the (so-called) forgetful functor $\mathcal{Z}(\mathcal{C}) \to \mathcal{C}$. Then,
\[ F(b_i) = \sum_j F_{i,j} a_j, \]
where $F_{i,j}$ are nonnegative integers. Hence,
\begin{equation}
m_i = \sum_j F_{i,j}d_j. \label{eq:rows}
\end{equation} 

There is an additive morphism $I: {\rm R} \to {\rm ZR}$ (not preserving $\FPdim$, so not multiplicative) induced by the adjoint of the forgetful functor. As a matrix, $I$ is just the transpose of $F$, i.e.,
\[ I(a_j) = \sum_i F_{i,j} b_i. \]

The $r \times n$ matrix associated with $I$ is commonly referred to as the \emph{induction matrix}. It satisfies several arithmetic properties, which imply that, for a given fusion ring, only finitely many induction matrices are possible. When the rank is sufficiently small, \Normaliz{} can be used to classify them (see~\S\ref{sub:NormInd}).

\begin{remark}
There can be zero, one or several possible induction matrices. If none, then the fusion ring ${\rm R}$ is excluded from categorification, which is very useful. If there are induction matrices but no ${\rm ZR}$ compatible with them, then R is excluded as well from categorification. Idem if there are compatible {\rm ZR} but no modular data. 
\end{remark}

In general, for a given fusion ring ${\rm R}$, several ${\rm ZR}$ are possible, and several $n$ ( = rank(${\rm ZR}$)) are possible. Hence, $n$ is also a variable. Note that \cite[Proposition 9.2.2]{EGNO15} states that for all $j$,
\[ F(I(a_j)) = \sum_t a_t a_j a_{t^*}. \]
But
\[ F(I(a_j)) = \sum_k \left( \sum_i F_{i,j} F_{i,k} \right) a_k, \]
and
\[ \sum_t a_t a_j a_{t^*} = \sum_k \left( \sum_{u,t} N_{t,j}^u N_{u,t^*}^k \right) a_k. \]
We get the following equation:
\begin{equation}
\sum_i F_{i,j} F_{i,k} = \sum_{u,t} N_{t,j}^u N_{u,t^*}^k. \label{eq:high}
\end{equation}
Note that
\[
F(I(a_1)) = \sum_t a_t a_{t^*} = \sum_k \left( \sum_t N_{t,t^*}^k \right) a_k,
\]
so the left multiplication matrix for $F(I(a_1))$ is
\begin{equation}\left( \sum_{t,k} N_{t,t^*}^k N_{k,l}^u \right) _{u,l}. \label{eq:eigenval}
\end{equation} 
This matrix has eigenvalues $n_V f_V$ with multiplicity $n_V^2$, where $f_V$ denotes the formal codegree of $V \in \mathrm{Irr}({\rm R}_{\mathbb{C}})$ and $n_V = \dim(V)$; see \S\ref{sub:formal}. As recalled in \S\ref{sub:Drin}, if ${\rm R}$ is a Grothendieck ring, then it is also a Drinfeld ring, and in this case, each $f_V$ is an integer dividing $\FPdim({\rm R})$, since ${\rm R}$ is integral. Note that $s := |\mathrm{Irr}({\rm R}_{\mathbb{C}})| \le r$, with equality if and only if ${\rm R}$ is commutative. According to \cite[Theorem 2.13]{Os15}, the set $\mathrm{Irr}({\rm R}_{\mathbb{C}})$ embeds into $\mathcal{O}(\mathcal{Z}(\mathcal{C}))$, and hence into the basis of $\mathrm{ZR}$ in our notation. Let us label the elements of $\mathrm{Irr}({\rm R}_{\mathbb{C}})$ as $V_i$, such that it maps to $b_i$. Define $f_i := f_{V_i}$, $n_i:=n_{V_i}$ and $m_i := \FPdim({\rm R})/f_i$, for $1 \le i \le s$. Finally, we have:
\begin{itemize}
\item $F(b_1) = a_1, \quad \text{so } F_{1,j} = \delta_{1,j},$
\item $F_{i,1} = n_i, \quad \text{for all } i \in \{1, \ldots, s\},$
\item $F_{i,1} = 0, \quad \text{for all } i \in \{s+1, \ldots, n\}.$
\end{itemize}
These last two identities also follow from \cite[Theorem 2.13]{Os15}, and will be abbreviated below as $F_{i,1} = n_i \delta_{i \leq s}$.

%
\subsection{Parameters, variables, and relations} \label{sub:sum}

\subsubsection{Parameters}
\begin{itemize}
  \item Fusion data $(N_{i,j}^k)$ of rank $r$;
  \item $d_i := \FPdim(a_i) = \text{norm of the matrix } (N_{i,j}^k)_{k,j}$, assumed to be positive integers;
  \item $\FPdim({\rm R}) = \sum_i d_i^2$;
  \item $n_jf_j = \text{eigenvalues of multiplicity $n_j^2$ of the matrix } \left( \sum_{t,k} N_{t,t^*}^k N_{k,l}^u \right)_{u,l}$;
  \item $s = |\mathrm{Irr}({\rm R}_{\mathbb{C}})| \le r$.
\end{itemize}

\subsubsection{Variables}
\begin{itemize}
  \item Rank $n$ of ${\rm ZR}$;
  \item $n \times r$ matrix $(F_{i,j})$;
  \item $m_i = \sum_j F_{i,j} d_j$, with $i \in \{1,\ldots,n\}$.
\end{itemize}

\subsubsection{Relations} \label{subsub:rel}
\begin{enumerate}
  \item \label{rel1} $F_{i,j}$ are nonnegative integers;
  \item \label{rel2} For all $j,k \in \{1,\ldots,r\}$, $\sum_i F_{i,j} F_{i,k} = \sum_{u,t} N_{t,j}^u N_{u,t^*}^k$;
  \item \label{rel3} For all $i \in \{1,\ldots,n\}$, $m_i$ is a positive integer dividing $\FPdim({\rm R})$;
  \item \label{rel4} $\sum_i m_i^2 = \FPdim({\rm R})^2$;
  \item \label{rel5} For all $i \in \{1,\ldots,s\}$, $f_i$ are positive integers dividing $\FPdim({\rm R})$;
  \item \label{rel6} For all $i \in \{1,\ldots,s\}$, $m_i = \FPdim({\rm R})/f_i$;
  \item \label{rel7} For all $j \in \{1,\ldots,r\}$, $F_{1,j} = \delta_{1,j}$;
  \item \label{rel8} For all $i \in \{1,\ldots,n\}$, $F_{i,1} = n_i\delta_{i \le s-1}$.
\end{enumerate}
The last three identities above allow us to directly determine $r + n - 1$ variables. Next, we focus on the subsystem of $r + s - 2$ linear Diophantine equations in $(r-1)(s-1)$ variables, each with positive coefficients, corresponding to the \textbf{lower part of the induction matrix}, specifically its $s \times r$ submatrix. 
For all $i \in \{2, \dots, s\}$ and $j \in \{2, \dots, r\}$, 
\begin{equation}
\sum_{i=2}^{s} n_i F_{i,j} = \sum_{t=1}^{r} N_{t, t^*}^j, \label{eq:col}
\end{equation}
$$
\sum_{j=2}^{r} d_j F_{i,j} = \frac{\FPdim({\rm R})}{f_i} - n_i.
$$
To prove the first identity, apply equation (\ref{rel2}) with $k = 1$. The term with $i = 1$ on the left-hand side vanishes due to equation (\ref{rel7}), since $j > 1$. Furthermore, equation (\ref{rel8}) gives $F_{i,1} = n_i$ for $i \le s$, and zero otherwise. On the right-hand side, note that $N_{s, t^*}^1 = \delta_{s,t}$, and by Frobenius reciprocity, $N_{t, j}^t = N_{t^*, t}^j$. Finally, we have $\sum_t N_{t^*, t}^j = \sum_t N_{t, t^*}^j$, which follows by a change of variable, replacing $t$ with $t^*$.
For the second identity, apply equation (\ref{rel6}). We have $m_i = \sum_j F_{i,j} d_j$, while equation (\ref{rel8}) gives $F_{i,1} d_1 = n_i$, since $i \le s$. Substituting this yields the stated expression.

%

\subsection{Ring morphism} \label{sub:ring}

The forgetful functor from $\mathcal{Z}(\mathcal{C})$ to $\mathcal{C}$ is a tensor functor. Consequently:
\[ F(b_{i^*}) = F(b_i)^* \text{ which implies } F_{i^*, j} = F_{i, j^*}, \text{ which could restrict the possible dualities}, \]
and
\[ F(b_i b_j) = F(b_i) F(b_j). \]

Given that $(M_{i,j}^k)$ represents the fusion data of ${\rm ZR}$, we have:
\[ F(b_i b_j) = F( \sum_k M_{i,j}^k b_k ) = \sum_k M_{i,j}^k F(b_k) = \sum_k M_{i,j}^k \left( \sum_t F_{k,t} a_t \right) = \sum_t \left( \sum_k M_{i,j}^k F_{k,t} \right) a_t. \]

On the other hand:
\begin{align*}
F(b_i) F(b_j) = \left( \sum_l F_{i,l} a_l \right) \left( \sum_u F_{j,u} a_u \right) &= \sum_{l,u} F_{i,l} F_{j,u} (a_l a_u) \\
&= \sum_{l,u} F_{i,l} F_{j,u} \left( \sum_t N_{l,u}^t a_t \right) = \sum_t \left( \sum_{l,u} F_{i,l} F_{j,u} N_{l,u}^t \right) a_t.
\end{align*}

Therefore, for all $i, j, t$ , we have:
$$
\sum_k M_{i,j}^k F_{k,t} = \sum_{l,u} F_{i,l} F_{j,u} N_{l,u}^t.
$$

These additional equations must be imposed in order to classify all possible fusion rings~$\mathrm{ZR}$ of type~$(m_i)$ that are compatible with the matrix~$(F_{i,j})$.

\section{Algorithmic details} \label{sec:AlgoDetails}

This section details the algorithms used for classifying Egyptian fractions (Section~\ref{sub:EgyAlgo}) with \SageMath{}, fusion rings (Section~\ref{sub:Normaliz}) and induction matrices (Section~\ref{sub:NormInd}) with \Normaliz{}.

\subsection{Efficient generation of Egyptian fractions} \label{sub:EgyAlgo}

To efficiently generate Egyptian fractions of a fixed length $r$, we employed \SageMath{}'s {\tt MapReduce} functionality, which provides easy parallelization for enumerating elements of a \emph{recursively enumerated set} (RES) with a forest structure. 

Our ultimate goal is to generate all Egyptian fractions
\[
\sum_{i=1}^r \frac{1}{f_i} = 1.
\]
To this end, we consider the RES consisting of partial Egyptian fractions
\[
\sum_{i=1}^{\ell} \frac{1}{f_{r+1-i}} = s,
\]
of length $\ell \leq r$, subject to the conditions
\[
1 \leq f_r \leq f_{r-\ell+1} \leq \dots \leq f_{r-\ell}, \quad s \leq 1,
\]
denoted by the tuple $(s; f_r, \dots, f_{r-\ell+1})$. Such partial Egyptian fractions correspond to nodes in a forest. The roots are $(\tfrac{1}{k};k)$ for $k=1,2,\dots,r$, and a node $(s;f_r,\dots,f_{r-\ell+1})$ has children
\[
\Bigl(s+\tfrac{1}{f_{r-\ell}};\; f_r,\dots,f_{r-\ell+1},f_{r-\ell}\Bigr),
\]
where the new denominator $f_{r-\ell}$ satisfies
\[
f_{r-\ell+1} \leq f_{r-\ell} \leq \frac{r-\ell}{1-s}.
\]
Nodes with $s=1$ or $\ell=r$ have no children.

\medskip

Several optimizations make this approach more practical. In particular, for each node with $\ell=r-2$ and rational sum $s=\tfrac{p}{q}$, the last two denominators $f_2$ and $f_1$ (which usually have the widest search ranges) can be determined directly from
\[
\frac{p}{q} + \frac{1}{f_2} + \frac{1}{f_1} = 1,
\]
which is equivalent to
\begin{equation}\label{eq:last_two_f}
\bigl((q-p)f_2 - q\bigr)\cdot \bigl((q-p)f_1 - q\bigr) \;=\; q^2.
\end{equation}
Thus the problem reduces to iterating over the divisors $d \mid q^2$ with
\(
d \equiv q \pmod{q-p}.
\)

We further incorporated the divisibility condition $f_i \mid f_1$ (motivated by the requirement that formal codegrees divide the global dimension $\FPdim$). This allows stronger pruning of the search and, combined with \eqref{eq:last_two_f}, implies
\(
\bigl((q-p)f_2 - q\bigr) \mid q.
\)

\medskip

Our implementation also supports additional constraints on the generated Egyptian fractions, such as requiring odd denominators only, pairwise distinct denominators, bounded denominators, or denominators with bounded prime factors. For applications to MNSD Drinfeld rings, we generalized the method to Egyptian fractions with prescribed numerators (taken to be $2,2,\dots,2,1$ in the MNSD case). Further implementation details are available in the GitHub repository~\cite{MaxScripts}.

\subsection{The fusion ring solver in \Normaliz{}}\label{sub:Normaliz}

In the following we explain the algorithmic ideas behind the fusion ring solver in \Normaliz{}  \cite{Norma}. See the manual \cite{NorManual} for input and output formats and the full range of options.

\Normaliz{} \cite{Norma} is a software package for computations in discrete convex geometry and toric algebra. One of its core tasks is to compute lattice points in a polytope $P \subset \RR^n$ defined by a system of linear equations and inequalities with rational coefficients. (\Normaliz{} can also handle polytopes defined over real algebraic number fields.)  

The computation of fusion rings of a given type and duality amounts to finding the lattice points $(N_{i,j}^k)$ in such a polytope that additionally satisfy the quadratic associativity equations from Definition~\ref{def:fu}. Concretely, our fusion data $(N_{i,j}^k)$ must satisfy the following conditions:
\begin{itemize}
\item (NonNeg)  \ \ $N_{i,j}^k$ is a nonnegative integer, for all $i,j,k$.
\item (DimEq) \ \  $d_i d_j = \sum_k N_{i,j}^k d_k$, for all $i,j$.
\item (Assoc) \ \  $\sum_s N_{i,j}^s N_{s,k}^t = \sum_s N_{j,k}^s N_{i,s}^t$, for all $i,j,k,t$. 
\end{itemize}

The conditions (Unit), (Dual), and (Anti-involution) from Definition~\ref{def:fu} are used to replace certain variables $N_{i,j}^k$ with fixed values $0$ or $1$, and to impose the identification $N_{i,j}^k = N_{j^*,i^*}^{k^*}$. This identification is further extended by Proposition~\ref{prop:FrobRec}. Note that in \Normaliz{} the basis vectors of a fusion ring are indexed starting at $0$, so that $b_0$ represents the unit element.

The default algorithm in \Normaliz{} for computing lattice points in polytopes $P \subset \RR^n$ is the standard \emph{project-and-lift} method, which proceeds inductively on the ambient dimension $n$. For $n=1$, one must simply find the interval in $\RR$ defined by the equations and inequalities. In the general case, the projection $P' \subset \RR^{n-1}$ is the image of $P$ under elimination of the $n$-th coordinate. Assuming the lattice points of $P'$ are known, the lattice points of $P$ are precisely those integral points in $P$ that project to some $x' \in P'$. The fiber over $x'$ is an interval in $\RR$. Figure~\ref{fig:proj} illustrates this idea.

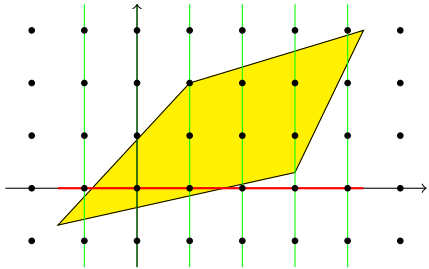
\begin{figure}[hbt]
\begin{center}
	\begin{tikzpicture}[scale=0.7]	
		\filldraw[yellow] (-1.5,-0.7) -- (1,2) -- (4.3,3) -- (3,0.3) -- cycle;
		\draw (-1.5,-0.7) -- (1,2) -- (4.3,3) -- (3,0.3) -- cycle;
		
		\foreach \x in {-1,...,4} {
			\draw[green] (\x,-1.5) -- (\x,3.5);
		}
		\draw[->] (-2.5,0) -- (5.5,0);
		\draw[->] (0,-1.5) -- (0,3.5);
		\draw[color=red,thick] (-1.5,0) -- (4.3,0);
		
		\foreach \x in {-2,...,5}
		\foreach \y in {-1,...,3} {
			\filldraw[fill=black] (\x,\y) circle (1.5pt);
		}
	\end{tikzpicture}
\end{center}
\caption{Illustration of the project-and-lift method.}
\label{fig:proj}
\end{figure} 

As simple as this idea is, it becomes infeasible for fusion rings because computing $P'$ requires pure Fourier--Motzkin elimination, i.e.\ eliminating variables from a system of linear constraints. Fortunately, there is a way around this. The algorithm admits relaxation: it suffices to compute an \emph{overpolytope} $P'' \supset P'$. For example, omitting a summand $d_k N_{i,j}^k$ in (DimEq) yields the valid inequality
\[
\sum_k d_k N_{i,j}^k \le d_i d_j.
\]
Together with the nonnegativity conditions (NonNeg), these truncated inequalities define the polytope $P''$. We refer to this approach as \emph{coarse projection}.

Another useful feature of (DimEq) is that they involve only a small number of variables. This makes \emph{patching} possible. Each such equation is a \emph{patch}. Restricted to the variables with nonzero coefficients, the patch admits \emph{local solutions}. A \emph{global solution} is then obtained by patching together local solutions that agree on overlaps.  

\Normaliz{} does not compute all local solutions beforehand. Instead, it inserts patches successively: at each step, the existing partial solutions are extended along the overlap with the next patch. Without the highly selective quadratic equations (Assoc), this process would typically cause a combinatorial explosion. To prevent this, partial solutions are discarded as soon as they violate any quadratic equation whose support is contained in the patches considered so far.

Some additional key aspects:
\begin{itemize}
\item Choosing a good insertion order of patches is crucial. The goal is to make associativity equations applicable as early as possible, but one should also avoid inserting ``bad'' patches with very large right-hand sides $d_i d_j$ too soon. \Normaliz{} provides options to balance these considerations.
\item ``Look-ahead'' techniques are employed by coarsening equations to congruences modulo coefficients. This reduces the number of variables and often applies to earlier patches. Similarly, equations can be relaxed to inequalities in fewer variables.
\item The system of equations is heavily overdetermined. Hence \emph{heuristic minimization} is applied: quadratic equations that hold for a sufficiently large number of partial solutions are discarded, since they no longer provide selective power.

\item The computation of fusion rings in \Normaliz{} is in principle a tree search. \Normaliz{} uses a mixture of breadth first and depth first strategies.
\end{itemize}

Both coarse projection and patching are particularly effective for $0$--$1$ problems.

Two sets of fusion data define isomorphic fusion rings if and only if they differ by a permutation $\pi$ of $\{b_1,\dots,b_r\}$ that commutes with duality and preserves the Frobenius--Perron dimensions. We are only interested in pairwise nonisomorphic fusion rings and adopt the convention that the lexicographically largest defining data serve as the \emph{normal form}. The algorithm discards partial solutions whenever it becomes clear that they cannot extend to a normal form.

For exceptionally large fusion data, we have used the high-performance cluster at Osnabrück (50 nodes, 64 threads each, 1~TB memory) with a static splitting strategy, allowing multi-round computations by successive refinement.

Table~\ref{tab:times} summarizes representative computation times. Here $\#$var, $\#$aut, and $\#$fus denote the number of variables, automorphisms, and fusion rings up to isomorphism, respectively. All timings were obtained on a PC with a Ryzen 900 CPU. Dualities are the identity, except for the last one, the type and duality of the character ring of $M_{12}$ (exponents count multiplicities, and the basic elements of $\FPdim=16$ are dual to each other). This example was one of the primary challenges in developing the \Normaliz{} algorithm. The first four examples are the smallest-rank exotic simple integral fusion rings mentioned in the introduction. The rank 12 example arose in a different context.

\begin{table}[hbt]
\begin{tabular}{|l|r|r|r|r|r|c|}
\hline
Type & Rank & FPdim & $\#$var & $\#$aut & $\#$fus & Time \\
\hline
$[1,11,14,16]$ & 4 & 574 & 10 & 1 & 1 & 0.002 s \\
\hline
$[1,9,10,11,21,24]$ & 6 & 1320 & 35 & 1 & 3 & 0.01 s \\ 
\hline
$[1,5,5,5,6,7,7]$ & 7 & 210 & 56 & 12 & 2 & 0.03 s \\
\hline
$[1,7,7,7,7,8,9,9,9]$ & 9 & 504 & 120 & 144 & 2 & 0.3 s \\
\hline
$[1,1,2,3,3,6,6,8,8,8,12,12]$ & 12 & 576 & 286 & 48 & 199 & 3.5 s \\
\hline
$[1,11^2,16^2,45,54,55^3,66,99,120,144,176]$ & 15 & 95040 & 546 & 24 & 12 &  59:20 m\\ 
\hline
\end{tabular}
\caption{Computation times.}
\label{tab:times}
\end{table}

\Normaliz{} can also be run in restricted modes. The option \texttt{NoCoarseProjection} disables coarse projection, while \texttt{NoPatching} disables patching. The only example in Table~\ref{tab:times} that runs reasonably with \texttt{NoCoarseProjection} is $[1,11,14,16]$. The option \texttt{NoPatching} works for $[1,9,10,11,21,24]$ and $[1,5,5,5,6,7,7]$, and in fact runs them slightly faster than the full algorithm, likely by avoiding the overhead of patch management. For $[1,7,7,7,7,8,9,9,9]$, however, \texttt{NoPatching} becomes ineffective.  

Finally, automorphisms can be ignored altogether by invoking \texttt{LatticePoints}, in which case \Normaliz{} simply enumerates all lattice points. For example, for $[1,1,2,3,3,6,6,8,8,8,12,12]$ this increases the computation time to $10.5$ seconds and produces $1659$ lattice points.

\subsection{Induction matrices via \Normaliz{}} \label{sub:NormInd}

A closer look at \S\ref{sub:sum} shows that computing induction matrices essentially amounts to finding lattice points in polytopes, although in a rather intricate way.

Throughout this subsection, rows and columns refer to the matrix $I$: $F_{i,j}$ denotes the entry in row $i$, column $j$. For details on the output format and options of \Normaliz{}, see the manual~\cite{NorManual}.

The first preparatory step in the computation of $I$ is straightforward. We check whether the fusion ring $R$ is commutative. Then we compute the eigenvalues of matrix~\eqref{eq:eigenval} and their multiplicities (as roots of the characteristic polynomial). Since the list of candidates is finite, we can simply test all of them.

In the commutative case, the multiplicity of an eigenvalue $f$ counts how often $\FPdim(R)/f$ appears as a value of some $m_i$ in the lower part of the potential induction matrices.

The noncommutative case is more delicate since $\mathrm{Irr}(R_{\mathbb{C}})$ is unknown. We only know the eigenvalues and their multiplicities, and must recover $n_j$ and $f_j$ for $j=1,\dots,s$. However, for ranks $\le 8$, Lemma~\ref{lem:NCmin} guarantees a unique choice for each $n_j$. This allows us to determine $s$, the number of rows in the lower part, and the corresponding $m_i$.

We now address the computation of the lower part. The variables are the entries $F_{i,j}$ with $i=1,\dots,s$ and $j=1,\dots,r$. The entries $F_{i,1}$ and $F_{1,j}$ are known. The rows must satisfy the linear equations~\eqref{eq:rows}, and the columns must satisfy~\eqref{eq:col}. This system of equations is fed into the project-and-lift algorithm of \Normaliz{}, together with the nonnegativity constraints on $F_{i,j}$. Coarse projection and patching, introduced in \S\ref{sub:Normaliz}, apply here as well.

For the higher part (rows beyond the lower part), the difficulty is that $n$ is not known. To proceed, we consider all potential rows given by solutions to
\[
\sum_j d_j F_{i,j} = m_i \quad \text{for all } i > s,
\]
with $F_{i1}=0$ when $i > s$. For each divisor $m_i$ of $\FPdim(R)$, this is again a lattice-point problem. Coarse projection applies, but no patching is needed. Let $H_k$, $k=1,\dots,h$, denote the list of all such representations of the divisors $m_i$. The unknowns are the multiplicities $\mu_k$ with which the rows $H_k$ occur. The number $h$ can be very large, which limits the computation. The constraints are:
\begin{enumerate}
  \item[(i)] $\mu_k \ge 0$ for all $k$, and
  \item[(ii)] linear equations: 
  \begin{enumerate}
  \item[(a)] one equation ensuring that the total Frobenius--Perron dimension of the center is $\FPdim(R)^2$, and 
  \item[(b)] the system derived from~\eqref{eq:high} for $i,j=1,\dots,r$, taking into account the contribution of the lower part.
  \end{enumerate}
\end{enumerate}
The crucial point is that the total contribution of rows equal to $H_k$ is linear in~$\mu_k$.

\Normaliz{} can also check ring homomorphisms between fusion rings, including those from \S\ref{sub:ring}. To do this, one specifies the image ring and the matrix of the homomorphism as a $\ZZ$-linear map. The potential preimages are determined by type and duality as usual. Requiring that the given map is a ring homomorphism imposes additional linear equations on the fusion data. As a result, only those fusion rings admitting the specified homomorphism are obtained.

\begin{example} \label{ex:IndMat}
We illustrate the computation with the example in \cite[\S\ref{sub:Rank4}(\ref{1126})]{appendices}, excluded in the proof of Lemma~\ref{lem:1126}. To generate the input file via \SageMath{} for the fusion ring of type \texttt{[1,1,2,6]} with duality \texttt{[0,1,2,3]}, run:
\begin{verbatim}
sage: %attach TypeToNormaliz.sage
sage: NormalizInduction([1,1,2,6],[0,1,2,3])
\end{verbatim}

Running \Normaliz{} on this input produces a \texttt{.ind} file listing all possible induction matrices for each fusion ring. In this case there is only one fusion ring, with the induction matrix:
\[
\begin{bmatrix}
1 & 0 & 0 & 1 & 0 & 0 & 0 & 0 & 0 & 0 & 0 & 1 & 0 & 0 & 1 & 0 \\
0 & 1 & 0 & 1 & 0 & 0 & 0 & 0 & 0 & 0 & 0 & 1 & 0 & 0 & 0 & 1 \\
0 & 0 & 1 & 2 & 0 & 0 & 0 & 0 & 0 & 0 & 0 & 0 & 1 & 1 & 1 & 1 \\
0 & 0 & 0 & 0 & 1 & 1 & 1 & 1 & 1 & 1 & 1 & 2 & 2 & 2 & 3 & 3
\end{bmatrix}
\]
followed by the type of the corresponding potential Drinfeld center:
\[
[1,1,2,6,6,6,6,6,6,6,6,14,14,14,21,21].
\]
\end{example}


\section{Exclusions via induction matrices} \label{sec:ExcMod} \label{sec:Indu}

\begin{remark} \label{rk:collected}
All computations of induction matrices in this paper—carried out as illustrated in Example \ref{ex:IndMat} using our new \Normaliz{} feature—are documented in the file \verb|InvestInduction.txt|, located in the \verb|Data/InductionMatrices| directory of~\cite{CodeData}.
\end{remark}

\begin{lemma} \label{lem:NoInd}
A fusion ring with any of the following type and duality structures admits no induction matrix, and therefore no categorification as a fusion category over~$\mathbb{C}$.
\begin{itemize}
\item $[1,1,4,4,6]$, $[0,1,3,2,4]$;
\item $[1,1,1,6,9]$, $[0,2,1,3,4]$;
\item $[1,1,5,7,8]$, $[0,1,2,3,4]$;
\item $[1,1,2,3,15]$, $[0,1,2,3,4]$;
\item $[1,1,2,9,15]$, $[0,1,2,3,4]$;
\item $[1,1,2,8,8,10]$, $[0,1,2,4,3,5]$;
\item $[1,1,2,8,8,14]$, $[0,1,2,4,3,5]$;
\item $[1,1,1,10,11,14]$, $[0,2,1,3,4,5]$;
\item $[1,1,8,10,10,14]$, $[0,1,2,4,3,5]$;	
\item $[1,1,1,1,1,1,3,3]$, $[0,1,2,3,5,4,7,6]$;
\item $[1,1,1,1,1,1,6,6]$, $[0,1,2,3,5,4,6,7]$.
\end{itemize}
\end{lemma}
\begin{proof}
For each case, the computation of all possible induction matrices yields no solution (see Remark~\ref{rk:collected}).
\end{proof}

\begin{lemma} \label{lem:NoInd2}
A fusion ring of type~$[1,1,1,1,1,1,3,3]$, duality~$[0,1,2,3,5,4,6,7]$, and multiplicity one admits no induction matrix, and therefore no categorification as a fusion category over~$\mathbb{C}$.
\end{lemma}

\begin{proof}
There are exactly two fusion rings with this type and duality. The one in~\cite[\S\ref{subsub:1FrobNCRank8}(\ref{[1,1,1,1,1,1,3,3]b})]{appendices} has multiplicity two and admits many induction matrices; the other, in~\cite[\S\ref{subsub:1FrobNCRank8}(\ref{[1,1,1,1,1,1,3,3]a})]{appendices}, has multiplicity one and admits none (see Remark~\ref{rk:collected}). 
\end{proof}

\begin{lemma} \label{lem:1126}
There exists no integral fusion category~$\mathcal{C}$ of rank $4$, $\FPdim$ $42$, and type $[1,1,2,6]$.
\end{lemma}
\begin{proof}
There is a unique possible induction matrix (see Remark~\ref{rk:collected}), embedding~$\mathcal{C}$ into its Drinfeld center~$\mathcal{Z}(\mathcal{C})$, which has rank~$16$ and type $t = [[1,2], [2,1], [6,8], [14,3], [21,2]].$ This embedding induces a braiding on~$\mathcal{C}$, but no premodular fusion category of type~$[1,1,2,6]$ exists, by~\cite[Theorem 4.11]{B16}.
\end{proof}

\begin{remark} \label{rk:alternative}
Here is an alternative to the last sentence of the proof of Lemma~\ref{lem:1126}. Since~$\FPdim(\mathcal{C}) = 42 = 2 \times 3 \times 7$, the fusion category~$\mathcal{C}$ must be group-theoretical by~\cite[Theorem 9.2]{ENO11}; that is, it is Morita equivalent to~$\VVec(G, \omega)$ for some finite group~$G$ of order~$42$ and some $3$-cocycle~$\omega$. Consequently,~$\mathcal{Z}(\mathcal{C})$ is braided equivalent to~$\mathcal{Z}(\VVec(G, \omega))$, by~\cite[Proposition 8.5.3]{EGNO15}. Such Drinfeld centers are characterized by the existence of a Lagrangian subcategory~$\Rep(G)$, according to~\cite[Proposition 9.13.5]{EGNO15}. However, there are only six non-isomorphic finite groups of order~$42$, and none of them have character degrees covered by the type~$t$ of~$\mathcal{Z}(\mathcal{C})$, leading to a contradiction.
\end{remark}


\begin{lemma} \label{lem:A4Braid}
Any fusion category of rank~$4$, $\FPdim$~$12$, and type~$[1,1,1,3]$ admits a braiding.
\end{lemma}

\begin{proof}
There are exactly two possible induction matrices (see Remark~\ref{rk:collected}). The first gives rise to a Drinfeld center of rank~$11$ and type~$[1,1,1,3,4,4,4,4,4,4,6]$, which is excluded as the type of a modular fusion category by~\cite{ABPP}. The second matrix embeds the fusion category into its Drinfeld center, which has rank~$14$ and type~$[1,1,1,3,3,3,3,3,4,4,4,4,4,4]$. The result follows.
\end{proof}

Any fusion category as in Lemma~\ref{lem:A4Braid} is Grothendieck equivalent to~$\Rep(A_4)$; see \cite[\S\ref{sub:Rank4}(\ref{[1,1,1,3][0,2,1,3]})]{appendices}.

\begin{lemma} \label{lem:AllBraidedR5}
Any fusion category of rank~$5$, $\FPdim$ $20$, and type~$[1,1,1,1,4]$ admits a braiding.
\end{lemma}

\begin{proof}
There are two possible fusion rings, each admitting a unique induction matrix (see Remark~\ref{rk:collected}). In both cases, the fusion category embeds into its Drinfeld center. The result follows.
\end{proof}

\begin{lemma} \label{lem:AllBraided[1,1,1,3,12]}
There is no fusion category of rank~$5$, $\FPdim \ 156$, and type~$[1,1,1,3,12]$.
\end{lemma}

\begin{proof}
There is a unique possible induction matrix (see Remark~\ref{rk:collected}), which embeds the fusion category into its Drinfeld center. The result then follows from~\cite[Theorem~I.1]{BO18}.
\end{proof}


\begin{lemma} \label{lem:11136}
There is no integral fusion category of rank $5$, $\FPdim \ 48 = 2^4 3$ and type $[1,1,1,3,6]$.
\end{lemma}
\begin{proof}   
There is a single possible fusion ring, see \cite[\S\ref{sub:Rank5}(\ref{[1,1,1,3,6]})]{appendices}. There are $12$ possible induction matrices (Remark \ref{rk:collected}), providing $12$ possible types $t$ for the Drinfeld center, whose modular partition are all as follows:  
$$[[1,1,1,3, \dots, 3, 6, \dots, 6, 12, \dots, 12, 24, \dots, 24],[16,16,16],[16,16,16]],$$
%
%
%
Let $\mathcal{D}$ be an integral modular fusion category of type $t$. It is easy to check that $\mathcal{D}_{pt}\cong \VVec(C_3)$ is Tannakian (just need to check the assumption of \cite[Theorem 8.2 (1)]{ABPP}).  Consider the functor $F: \mathcal{D}\to \mathcal{D}_{C_3}$, the de-equivariantization of $\mathcal{D}$ by $C_3$. Let $X$ be a $16$-dimensional simple object of $\mathcal{D}$. Since $\gcd(3,16)=1$, $F(X)=Y$ is also a simple object in $\mathcal{D}_{C_3}$ by \cite[Lemma 7.2]{NR16}. $\mathcal{D}_{C_3}$ admits a faithful $C_3$-grading with the trivial component $\mathcal{D}_{C_3}^0$.  Since the $\FPdim$ of $\mathcal{D}_{C_3}^0$ is $2^8$, $\mathcal{D}_{C_3}^0$ is nilponent and so is $\mathcal{D}_{C_3}$.  Hence $\FPdim(Y)^2=2^8$ divides the dimension of the $(\mathcal{D}_{C_3})_{ad}$, the  adjoint subcategory of $\mathcal{D}_{C_3}$, by \cite[Corollary 5.3]{GN08}. By \cite[Theorem 8.28]{ENO05}, $\mathcal{D}_{C_3}^0$ admits a faithful $C_2$-grading. Hence  $\FPdim((\mathcal{D}_{C_3})_{ad})$ is at most $2^7$. This contradicts the fact we have gotten that $\FPdim(Y)^2=2^8$ dividing the dimension of the $(\mathcal{D}_{C_3})_{ad}$.
\end{proof}

\begin{lemma} \label{lem:11266}
There is no integral fusion category~$\mathcal{C}$ of rank~$5$, with~$\FPdim(\mathcal{C}) = 78$ and type~$[1,1,2,6,6]$.
\end{lemma}

\begin{proof}
As in the proof of Lemma~\ref{lem:1126}, there is a unique induction matrix (see Remark~\ref{rk:collected}), embedding~$\mathcal{C}$ into its Drinfeld center~$\mathcal{Z}(\mathcal{C})$, which has rank~$36$ and type~$t = [[1, 2], [2, 1], [6, 28], [26, 3], [39, 2]]$. This embedding induces a braiding on~$\mathcal{C}$, but according to~\cite[Theorem I.1]{BO18}, there is no premodular fusion category of type~$[1,1,2,6,6]$.
\end{proof}
The same alternative proof as in Remark~\ref{rk:alternative} applies, since $78 = 2 \times 3 \times 13$, and none of the six finite groups of order~$78$ have character degrees covered by the type~$t$.


\begin{lemma} \label{lem:112210}
There is no integral fusion category~$\mathcal{C}$ of rank~$5$, with~$\FPdim(\mathcal{C}) = 110$ and type~$[1,1,2,2,10]$.
\end{lemma}

\begin{proof}
As in the proof of Lemma~\ref{lem:1126}, there is a unique induction matrix (see Remark~\ref{rk:collected}), embedding~$\mathcal{C}$ into its Drinfeld center~$\mathcal{Z}(\mathcal{C})$, which has rank~$28$ and type~$t = [[1, 2], [2, 2], [10, 12], [22, 10], [55, 2]]$. This embedding induces a braiding on~$\mathcal{C}$, but according to~\cite[Theorem I.1]{BO18}, there is no premodular fusion category of type~$[1,1,2,2,10]$.
\end{proof}
The same alternative proof as in Remark~\ref{rk:alternative} applies, since $110 = 2 \times 5 \times 11$, and none of the six finite groups of order~$110$ have character degrees covered by the type~$t$.


\subsubsection*{Strong Lagrange}

Let us conclude this section by providing an alternative exclusion process using the following result, which is stronger than Lagrange's theorem and is a consequence of~\cite[Remark 8.17]{ENO05}:

\begin{proposition} \label{prop:Lag+}
Let~$\mathcal{C}$ be a fusion category over~$\mathbb{C}$, and let~$\mathcal{D}$ be a fusion subcategory. Let~$\mathcal{M}$ be an indecomposable component of~$\mathcal{C}$ considered as a left~$\mathcal{D}$-module. Then~$\frac{\FPdim(\mathcal{M})}{\FPdim(\mathcal{D})}$ is an algebraic integer.
\end{proposition}

Proposition~\ref{prop:Lag+} provides a criterion for categorification (a necessary condition): from a fusion ring, one can classify the indecomposables for each fusion subring and verify whether the divisibility condition holds. A specific case of this is evident directly from the type:

\begin{corollary} \label{prop:Lag+type}
Let~$\mathcal{C}$ be a fusion category over~$\mathbb{C}$, and let~$[[d_1, n_1], [d_2, n_2], \dots, [d_s, n_s]]$ be its type, where~$d_1 = 1 < \cdots < d_s$. For each~$i$, the quantity~$\frac{n_i d_i^2}{n_1}$ is an algebraic integer.
\end{corollary}

\begin{proof}
Apply Proposition~\ref{prop:Lag+} to~$\mathcal{D} = \mathcal{C}_{pt}$.
\end{proof}

In particular, in the integral case,~$n_1$ divides~$n_i d_i^2$ for all~$i$. This can be applied, for example, to the fusion ring of type~$[1, 1, 5, 7, 8]$ in~\cite[\S\ref{sub:Rank5}(\ref{11578})]{appendices}.

\section{Group-theoretical models} \label{sec:grpth}

After recalling some basic facts about group-theoretical fusion categories \( \mathcal{C}(G, \omega, H, \psi) \) in \S\ref{sub:grpthbasic}, we explain in \S\ref{sub:r7d60nc} why the noncommutative Drinfeld ring of type $[1, 1, 1, 3, 4, 4, 4]$ is group-theoretical, and we provide a \GAP{} code to verify this automatically for categories of the form \( \mathcal{C}(G, 1, H, 1) \). In \S\ref{sub:schur}, we recall results involving the Schur multiplier that, under suitable conditions, allow a reduction to this specific case, which in turn enables us to exclude a fusion category of type $[1, 1, 1, 3, 3, 21, 21]$ in \S\ref{sub:r7d903}. Finally, in \S\ref{sub:NonIsaacs}, we exhibit the first non-Isaacs group-theoretical fusion category.

\subsection{Basics} \label{sub:grpthbasic}
Following \cite[\S 9.7]{EGNO15}, a fusion category is called \emph{group-theoretical} if it is Morita equivalent to pointed fusion category, \( \mathrm{Vec}(G, \omega) \) for some finite group \( G \) and 3-cocycle \( \omega \). Such a category is completely determined by a quadruple \( (G, \omega, H, \psi) \), where \( H \subseteq G \) is a subgroup and \( \psi \) is a 2-cochain on \( H \) satisfying \( d_2 \psi = \omega|_{H \times H \times H} \). It is denoted by \( \mathcal{C}(G, \omega, H, \psi) \), and it is always integral \cite[Remark 9.7.7]{EGNO15}. Furthermore, as shown in \cite[Theorem 9.2]{ENO11}, any integral fusion category of dimension \( pqr \), where \( p, q, r \) are distinct prime numbers, is necessarily group-theoretical.

Let \( R \) be a set of representatives for the double cosets \( H \backslash G / H \). According to \cite{Os03b}, \cite[\S 5.1]{GeNa09}, and \cite[Example 9.7.4]{EGNO15}, there is a bijection between the isomorphism classes of simple objects in \( \mathcal{C} \) and the isomorphism classes of pairs \( (g, \rho) \), where \( g \in R \) and \( \rho \) is an irreducible projective representation of the subgroup
\(
H^g = H \cap g H g^{-1}
\)
with 2-cocycle \( \psi^g \in H^2(H^g, \mathbb{C}^\times) \). That is, \( \rho \) is a simple object of \( \Rep_{\psi^g}(H^g) \). The cocycle \( \psi^g \) depends on \( \psi \) and \( \omega \), and is trivial when both \( \psi \) and \( \omega \) are trivial. The corresponding simple object is denoted \( X_{g, \rho} \).

According to \cite[Remark 2.3(ii)]{Nik08}, \cite[Proof of Theorem 5.1]{GeNa09}, and \cite[Theorem 6.1]{GS25},
\[
\FPdim(X_{g, \rho}) = [H : H^g] \cdot \dim(\rho) \  \text{ and } \   X_{g, \rho}^* = X_{g', \nu^*},
\]
where \( \{g'\} = R \cap (H g^{-1} H) \), and \( \nu \) is a simple object of \( \Rep_{\psi^{g'}}(H^{g'}) \) determined as follows. First, observe that the subgroups \( H^{g} \) and \( H^{g'} \) are conjugate. Indeed, since there exist \( h_1, h_2 \in H \) such that  
\(
g' = h_1 g^{-1} h_2,
\)
it follows that  
\(
H^{g'} = h_1 H^{g^{-1}} h_1^{-1}.
\)
Moreover, \( g^{-1} H^g g = H^{g^{-1}} \). For $h \in H^{g'}$, the simple object \( \nu \) is given by
\[
\nu(h) = \rho(g h_1^{-1} h h_1 g^{-1}). 
\]

Importantly, in general one cannot assume \( g' = g^{-1} \). For instance, this fails when \( (H,G) = (C_2, C_4) \).


\subsection{Type $[1,1,1,3,4,4,4]$} \label{sub:r7d60nc}
This subsection is devoted to demonstrating that the noncommutative Drinfeld ring of rank $7$, with $\FPdim$ 60 and type $[1,1,1,3,4,4,4]$ described in \cite[\S\ref{sub:NCRank7}(\ref{[1,1,1,3,4,4,4]})]{appendices}, is isomorphic to the Grothendieck ring of the group-theoretical fusion category $\mathcal{C}(A_5, 1, A_4, 1)$.
The set of double cosets
\[
R = A_4 \backslash A_5 / A_4 = \{ A_4, A_4 g A_4 \}, \quad g \notin A_4
\]
contains exactly two distinct double cosets, of sizes 12 and 48 respectively.

\textit{Trivial Double Coset:} The double coset $A_4$ corresponds to representations of the subgroup $A_4$, which has irreducible representations of dimensions $1, 1, 1,$ and $3$. This yields three simple objects of dimension 1 and one simple object of dimension 3.

\textit{Non-Trivial Double Coset:} For $g \notin A_4$, the stabilizer subgroup $A_4 \cap g A_4 g^{-1}$ is isomorphic to the cyclic group $C_3$, which has index 4 in $A_4$. The irreducible representations of $C_3$ are all one-dimensional; each such representation induces a simple object in $\mathcal{C}$ of dimension $4 \times 1 = 4$. Thus, this double coset contributes three simple objects of dimension 4.

This computation can be independently verified using the \GAP{}~\cite{gap} function \texttt{GroupTheoreticalType}, available in the file \texttt{GroupTheoretical.gap} located in the \verb|Code/GAP| directory of~\cite{CodeData}.
\begin{verbatim}
gap> Read("GroupTheoretical.gap");
gap> G:=AlternatingGroup(5);; H:=AlternatingGroup(4);;
gap> GroupTheoreticalType(G,H);
[ 1, 1, 1, 3, 4, 4, 4 ]
\end{verbatim}

With some extra work, one could write a script that computes all possible types for all choices of \( \omega \) and \( \psi \).

Finally, there are two Drinfeld rings of type $[1,1,1,3,4,4,4]$: a commutative one, say $\mathcal{R}_C$, and a noncommutative one, say $\mathcal{R}_{NC}$. To verify that the Grothendieck ring above is isomorphic to $\mathcal{R}_{NC}$, we observed that $\mathcal{R}_C$ has some formal codegrees equal to $6$. In contrast, the Drinfeld center $\mathcal{Z}(\VVec(A_5))$, which has type $$[1, 3, 3, 4, 5, 12, 12, 12, 12, 12, 12, 12, 12, 12, 12, 15, 15, 15, 15, 20, 20, 20],$$ contains no simple object with $\FPdim = 60/6 = 10$. 

Alternatively, since \( \mathcal{R}_C \) contains five self-dual basic elements, whereas \( \mathcal{R}_{NC} \) contains only three, the conclusion can also be verified by the following computation using \texttt{GroupTheoreticalTypeDuality}, a function available in the same \GAP{} file as above. This function also computes the duality.
\begin{verbatim}
gap> Read("GroupTheoretical.gap");
gap> GroupTheoreticalTypeDuality(G, H);
[ [ 1, 1, 1, 3, 4, 4, 4 ], [0, 2, 1, 3, 4, 5, 6] ]
\end{verbatim}
\subsection{Schur multiplier} \label{sub:schur}
According to \cite[Definition 11.12]{Isaacs}, the \emph{Schur multiplier} of a finite group \( G \) is the abelian group \( H^2(G, \mathbb{C}^\times) \), denoted \( M(G) \). As shown in \cite[Corollary 11.21]{Isaacs}, if a prime \( p \) divides \( |M(G)| \), then the Sylow \( p \)-subgroup of \( G \) is not cyclic. It follows that if all Sylow subgroups of \( G \) are cyclic, then \( M(G) \) is trivial. In particular, \( M(G) \) is trivial whenever \( G \) is cyclic or has square-free order.

\subsection{Type $[1,1,1,3,3,21,21]$} \label{sub:r7d903}
This subsection is devoted to proving that the Drinfeld ring of rank $7$, $\FPdim$ $903$ and type $[1,1,1,3,3,21,21]$ described in \cite[\S\ref{sub:NCRank7}(\ref{[1,1,1,3,3,21,21]})]{appendices}, does not admit a categorification.
Assume, for contradiction, that such a fusion category $\mathcal{C}$ exists. Since $903 = 3 \times 7 \times 43$, the category $\mathcal{C}$ must be group-theoretical (see \S\ref{sub:grpthbasic}), and hence of the form $\mathcal{C}(G, \omega, H, \psi)$, where $G$ is a finite group of order $903$ and $H$ is a subgroup of $G$.
Since $903$ is square-free, every subgroup $H^g = H \cap gHg^{-1}$ also has square-free order. As recalled in \S\ref{sub:schur}, this ensures that the Schur multiplier $M(H^g)$ is trivial. Consequently, all the $2$-cocycles $\psi^g$ are trivial. Therefore, using the notation from \S\ref{sub:grpthbasic}, the type of the category $\mathcal{C}(G, \omega, H, \psi)$ coincides with that of $\mathcal{C}(G, 1, H, 1)$.
The \GAP{} function \texttt{FindGroupSubgroup}, available in the same \GAP{} file as above, classifies all pairs of groups $G$ and subgroups $H$ such that the group-theoretical category $\mathcal{C}(G,1,H,1)$ has type \texttt{t}, and also computes the corresponding duality. When applied to the case $t = [1,1,1,3,3,21,21]$, the function returns no solutions.
\begin{verbatim}
gap> Read("GroupTheoretical.gap");
gap> FindGroupSubgroup([1,1,1,3,3,21,21]);
[  ]
\end{verbatim}
%

\subsection{A non-Isaacs group-theoretical category} \label{sub:NonIsaacs}

In this subsection, we exhibit the first known example of an integral fusion category that is not Isaacs, yet is group-theoretical.

Let \( R \) be a commutative fusion ring. Denote its character table by \( (\lambda_{i,j}) \), the basic Frobenius-Perron dimensions by \( (d_i) \), and the formal codegrees by \( (c_j) \), where \( d_1 = 1 \) and \( c_1 = \FPdim(R) \). We say that \( R \) is \emph{Isaacs} if
\(
\frac{\lambda_{i,j} c_1}{d_i c_j}
\)
is an algebraic integer for all \( i,j \). This notion has been extended to noncommutative fusion rings in \cite{BP25}. A pseudo-unitary fusion category is said to be Isaacs if its Grothendieck ring satisfies this property.

The following computation shows that the group-theoretical fusion category \( \mathcal{C}(G, 1, H, 1) \) has type \( [1,1,2,3,3,6] \) if and only if \( G = A_5 \) and \( H = S_3 \):
\begin{verbatim}
gap> FindGroupSubgroup([1,1,2,3,3,6]);
[ [ 5, "A5", 6, "S3", [ [ 1, 1, 2, 3, 3, 6 ], [ 0, 1, 2, 3, 4, 5 ] ] ] ]
\end{verbatim}

However, according to the classification, \cite[\S\ref{subsub:1FrobR6}(\ref{[1,1,2,3,3,6][0,1,2,3,4,5]})]{appendices} is the unique self-dual Drinfeld ring of type \( [1,1,2,3,3,6] \).
We will prove that it is not Isaacs. Its formal codegrees are \( [60, 15, 12, 4, 4, 3] \), and its character tables is:

\[
\left[
\begin{matrix}
1 & 1 & 1 & 1 & 1 & 1 \\
1 & 1 & 1 & -1 & -1 & 1 \\
2 & 2 & 2 & 0 & 0 & -1 \\
3 & -2 & 1 & 1 & -1 & 0 \\
3 & -2 & 1 & -1 & 1 & 0 \\
6 & 1 & -2 & 0 & 0 & 0
\end{matrix}
\right],
\]
We observe that
\[
\frac{\lambda_{4,2} c_1}{d_4 c_2} = \frac{-2 \times 60}{3 \times 15} = -\frac{8}{3},
\]
which is not an algebraic integer. Therefore, the fusion ring is not Isaacs. We conclude that the group-theoretical fusion category \( \mathcal{C}(A_5, 1, S_3, 1) \) is not Isaacs.

\section{Integral Drinfeld rings} \label{sec:General}

This section provides a summary of the computations used to classify all integral Drinfeld rings of rank at most~$5$ in general, and of higher ranks under additional assumptions. The table below displays the corresponding bounds and the number of Drinfeld rings identified in each case:

\begin{center}
\begin{tabular}{c|c|c|c|c}
\S & \textbf{Rank} & \textbf{Case} & \textbf{Bound on $\FPdim$} & \textbf{Number of Drinfeld rings} \\
\hline
\ref{sub:MainProof} & $\le 5$ & All & All & $29$ \\
\ref{subsub:1FrobRank6} & $6$ & $1$-Frobenius & All & $58$ \\
\ref{subsub:N1FrobRank6} & $6$ &  Non-$1$-Frobenius & $\le 200000$ & $88$ \\
\ref{subsub:1FrobRank7} & $7$ & $1$-Frobenius & $\le 100000$ & $241$ \\
\ref{subsub:N1FrobRank7} & $7$ & Non-$1$-Frobenius & $\le 5000$ & $113$ \\
\ref{sub:1FrobRank8} & $8$ & $1$-Frobenius & $\le 20000$ & $750$ \\
\ref{sub:1FrobRank9} & $9$ & $1$-Frobenius & $\le 2000$ & $1292$
\end{tabular}
\end{center}

A technical subtlety arises in the noncommutative case; we have deferred this discussion to \cite[\S\ref{sub:DivA}]{appendices}.

\subsection{Up to rank 5} \label{sub:MainProof} 
There are $1 + 1 + 3 + 14 + 147 = 166$ Egyptian fractions of length at most~$5$ (see~\cite{A002966}), but only $1 + 1 + 3 + 12 + 97 = 114$ of them satisfy the divisibility condition (see~\cite{A374582}); the full list is available in the \verb|Data/EgyptianFractionsDiv| folder of~\cite{CodeData}. These correspond to $1$, $1$, $3$, $9$, and $48$ distinct global $\FPdim$ (up to $1$, $2$, $6$, $42$, and $1806$, respectively), which in turn yield $1 + 1 + 2 + 7 + 208 = 219$ potential types.

Among these, only $27$ types admit fusion rings, giving rise to $36$ fusion rings in total, of which $29$ are Drinfeld rings. A detailed list—including the global $\FPdim$, type, duality, formal codegrees, and fusion data for each Drinfeld ring—is provided in~\cite[\S\ref{sub:UpToRank5}]{appendices}. For each case, either a concrete categorification is given or a reference is provided to a theoretical obstruction; all such exclusions are collected in~\S\ref{sec:ExcMod}.
Complete computational details and copy-pastable data can be found in the file \verb|GeneralUpToRank5.txt|, located in the \verb|Data/General| directory of~\cite{CodeData}. 

All relevant \SageMath{}~\cite{sage} functions are provided in the \verb|Code/SageMath| directory of~\cite{CodeData} and are explained in~\cite{ABPP}. For computations involving \Normaliz{}, see~\cite{NorManual}.

\subsection{Rank 6}  \label{sub:GeneralRank6}

There are $3462$ Egyptian fractions of rank $6$ (see \cite{A002966}). Among them, exactly $1568$ ones satisfy the divisibility assumption (see \cite{A374582}), corresponding to $492$ different possible global $\FPdim$ between $6$ and $3263442$, and $37694793$ possible types.  

\subsubsection{$1$-Frobenius case} \label{subsub:1FrobRank6} 

The number of $1$-Frobenius types is $1406$ only, but $40$ ones only admit fusion rings, $125$ fusion rings in total, and only $58$ ones are Drinfeld rings. The list is available in \cite[\S\ref{subsub:1FrobR6}]{appendices}. It contains a single simple item, the Grothendieck ring of $\Rep({\rm PSL}(2,7))$.
Copy-pastable data can be found in the file \verb|1FrobR6.txt|, located in the \verb|Data/General| directory of~\cite{CodeData}.

\subsubsection{Non-$1$-Frobenius case} \label{subsub:N1FrobRank6}

Without the $1$-Frobenius assumption, we already considered the $478$ first global $\FPdims$ (among $492$), those less than $200000$, themselves corresponding to $5597826$ possibles types, but $165$ types only admit fusion rings, $297$ fusion rings in total, and only $88$ ones are Drinfeld rings. Among them, exactly $32$ ones are non-$1$-Frobenius. They are listed in \cite[\S\ref{subsub:non1FrobR6}]{appendices}. Among them, there are two simple exotic ones. They are the first exotic simple integral non-$1$-Frobenius Drinfeld rings, and one of them is 3-positive (see \cite[\S\ref{subsub:non1FrobR6}(\ref{[1,9,10,11,21,24]})]{appendices}, and the comments in \S\ref{sub:ExoticIntro}).   
Copy-pastable data can be found in the file \verb|N1FrobR6d200000.txt|, located in the \verb|Data/General| directory of~\cite{CodeData}.

\begin{remark} 
With current technology, it should be feasible to achieve the remaining $14$ global $\FPdims$ necessary to complete the classification. However, this would require a herculean computational effort.
\end{remark}

\subsection{Rank 7} \label{sub:GeneralRank7} There are $294314$ Egyptian fractions of rank $7$ (see \cite{A002966}). Among them, exactly $76309$ satisfy the divisibility assumption (see \cite{A374582}), corresponding to $20655$ distinct possible global $\FPdim$ values between $7$ and $10650056950806$.

\subsubsection{1-Frobenius case} \label{subsub:1FrobRank7}

In the $1$-Frobenius case, we have already considered the $3370$ possible global $\FPdim \le 10^5$, which correspond to $60740$ possible types. Among these, only $183$ types admit fusion rings, yielding a total of $2066$ fusion rings, of which only $241$ are Drinfeld rings. These are available in the file \verb|1FrobR7d10^5.txt| within the \verb|Data/General| folder of \cite{CodeData}. Among them, exactly $5$ are simple, as listed in \cite[\S\ref{subsub:1FrobRank7A}]{appendices} where (\ref{[1,5,5,5,6,7,7]}) is the only simple exotic 3-positive one.
See Question~\ref{qu:F210indIntro} (and the paragraphs around it), motivated by discussions with Scott Morrison and Pavel Etingof.

\subsubsection{Non-1-Frobenius case} \label{subsub:N1FrobRank7}

Without the $1$-Frobenius assumption, we have already considered the $685$ possible global $\FPdim \le 5000$, which correspond to $1864563$ potential types. Among these, only $646$ types admit fusion rings, yielding a total of $2938$ fusion rings. Of these, just $284$ are Drinfeld rings, and exactly $113$ are non-$1$-Frobenius. These are listed in the file \verb|N1FrobR7d5000.txt|, located in the \verb|Data/General| folder of~\cite{CodeData}. Among the non-$1$-Frobenius Drinfeld rings, exactly $5$ are simple, as detailed in~\cite[\S\ref{subsub:N1FrobRank7A}]{appendices}. Only two of them are both exotic and $3$-positive.

\subsection{Rank 8}  \label{sub:1FrobRank8}
There are $159330691$ Egyptian fractions of rank~$8$ (see~\cite{A002966}). Among them, exactly $16993752$ satisfy the divisibility condition (see~\cite{A374582}), corresponding to $5792401$ distinct values of the global $\FPdim$ between~$8$ and~$113423713055421844361000443$. 

Without delving into details, by pushing \Normaliz{} to its limits on HPC, we classified all $750$ integral $1$-Frobenius Drinfeld rings of rank~$8$ with $\FPdim \leq 20000$ (as well as all $792$ with $\FPdim \leq 25000$, except for the unresolved type~$[1,44,49,55,55,56,56,70]$). The complete list is provided in the file \verb|1FrobR8d25000.txt|, located in the \verb|Data/General| directory of~\cite{CodeData}.

Among these, exactly seven are simple, but only one (beow)is both $3$-positive and exotic. It is isotype to~$\Rep(\mathrm{PSL}(2,11))$ and was ruled out from categorification by the zero spectrum criterion in~\cite{LPR2203}.

\begin{itemize} 
\item $\FPdim \ 660$, type $[1,5,5,10,10,11,12,12]$, duality $[0,2,1,3,4,5,6,7]$, fusion data:
{\fontsize{10}{12}\selectfont $$ 
\begin{smallmatrix} 1&0&0&0&0&0&0&0 \\ 0&1&0&0&0&0&0&0 \\ 0&0&1&0&0&0&0&0 \\ 0&0&0&1&0&0&0&0 \\ 0&0&0&0&1&0&0&0 \\ 0&0&0&0&0&1&0&0 \\ 0&0&0&0&0&0&1&0 \\ 0&0&0&0&0&0&0&1\end{smallmatrix} ,\ 
\begin{smallmatrix} 0&1&0&0&0&0&0&0 \\ 0&0&1&1&1&0&0&0 \\ 1&0&0&0&0&0&1&1 \\ 0&0&1&0&1&1&1&1 \\ 0&0&1&1&0&1&1&1 \\ 0&0&0&1&1&1&1&1 \\ 0&1&0&1&1&1&1&1 \\ 0&1&0&1&1&1&1&1\end{smallmatrix} ,\ 
\begin{smallmatrix} 0&0&1&0&0&0&0&0 \\ 1&0&0&0&0&0&1&1 \\ 0&1&0&1&1&0&0&0 \\ 0&1&0&0&1&1&1&1 \\ 0&1&0&1&0&1&1&1 \\ 0&0&0&1&1&1&1&1 \\ 0&0&1&1&1&1&1&1 \\ 0&0&1&1&1&1&1&1\end{smallmatrix} ,\ 
\begin{smallmatrix} 0&0&0&1&0&0&0&0 \\ 0&0&1&0&1&1&1&1 \\ 0&1&0&0&1&1&1&1 \\ 1&0&0&3&1&1&2&2 \\ 0&1&1&1&1&2&2&2 \\ 0&1&1&1&2&2&2&2 \\ 0&1&1&2&2&2&2&2 \\ 0&1&1&2&2&2&2&2\end{smallmatrix} ,\ 
\begin{smallmatrix} 0&0&0&0&1&0&0&0 \\ 0&0&1&1&0&1&1&1 \\ 0&1&0&1&0&1&1&1 \\ 0&1&1&1&1&2&2&2 \\ 1&0&0&1&3&1&2&2 \\ 0&1&1&2&1&2&2&2 \\ 0&1&1&2&2&2&2&2 \\ 0&1&1&2&2&2&2&2\end{smallmatrix} ,\ 
\begin{smallmatrix} 0&0&0&0&0&1&0&0 \\ 0&0&0&1&1&1&1&1 \\ 0&0&0&1&1&1&1&1 \\ 0&1&1&1&2&2&2&2 \\ 0&1&1&2&1&2&2&2 \\ 1&1&1&2&2&2&2&2 \\ 0&1&1&2&2&2&3&2 \\ 0&1&1&2&2&2&2&3\end{smallmatrix} ,\ 
\begin{smallmatrix} 0&0&0&0&0&0&1&0 \\ 0&1&0&1&1&1&1&1 \\ 0&0&1&1&1&1&1&1 \\ 0&1&1&2&2&2&2&2 \\ 0&1&1&2&2&2&2&2 \\ 0&1&1&2&2&2&3&2 \\ 1&1&1&2&2&3&2&3 \\ 0&1&1&2&2&2&3&3\end{smallmatrix} ,\ 
\begin{smallmatrix} 0&0&0&0&0&0&0&1 \\ 0&1&0&1&1&1&1&1 \\ 0&0&1&1&1&1&1&1 \\ 0&1&1&2&2&2&2&2 \\ 0&1&1&2&2&2&2&2 \\ 0&1&1&2&2&2&2&3 \\ 0&1&1&2&2&2&3&3 \\ 1&1&1&2&2&3&3&2\end{smallmatrix}
$$}
\item Formal codegrees: $[4,5,5,11,11,12,12,660]$,
\item Property: non-Isaacs, isotype to $\Rep({\rm PSL}(2,11))$,
\item Categorification: excluded in \cite{LPR2203}.
\end{itemize}

%

\subsection{Rank 9}  \label{sub:1FrobRank9}

Without going into detail—and \emph{without} pushing \Normaliz{} to its limits—we classified all $1292$ integral $1$-Frobenius Drinfeld rings of rank~$9$ with $\FPdim \le 2000$. The complete list is available in the file \verb|1FrobR9d2000.txt|, located in the \verb|Data/General| directory of~\cite{CodeData}.

Among these, exactly $9$ are simple, but only two (listed below) are both $3$-positive and exotic. They are isotype to $\Rep(G)$ for $G = \mathrm{PSL}(2,q)$ with $q = 8$ and $13$, respectively, and remain open for categorification.

\begin{enumerate}
\item $\FPdim \ 504$, type $[1, 7, 7, 7, 7, 8, 9, 9, 9]$, duality $[0,1,2,3,4,5,6,7,8]$, fusion data:
{\fontsize{7}{8}\selectfont $$ 
\begin{smallmatrix}1&0&0&0&0&0&0&0&0 \\ 0&1&0&0&0&0&0&0&0 \\ 0&0&1&0&0&0&0&0&0 \\ 0&0&0&1&0&0&0&0&0 \\ 0&0&0&0&1&0&0&0&0 \\ 0&0&0&0&0&1&0&0&0 \\ 0&0&0&0&0&0&1&0&0 \\ 0&0&0&0&0&0&0&1&0 \\ 0&0&0&0&0&0&0&0&1\end{smallmatrix} ,   \ 
\begin{smallmatrix}0&1&0&0&0&0&0&0&0 \\ 1&0&1&1&1&0&1&1&1 \\ 0&1&1&0&0&1&1&1&1 \\ 0&1&0&1&0&1&1&1&1 \\ 0&1&0&0&1&1&1&1&1 \\ 0&0&1&1&1&1&1&1&1 \\ 0&1&1&1&1&1&1&1&1 \\ 0&1&1&1&1&1&1&1&1 \\ 0&1&1&1&1&1&1&1&1\end{smallmatrix} ,   \ 
\begin{smallmatrix}0&0&1&0&0&0&0&0&0 \\ 0&1&1&0&0&1&1&1&1 \\ 1&1&0&1&1&0&1&1&1 \\ 0&0&1&1&0&1&1&1&1 \\ 0&0&1&0&1&1&1&1&1 \\ 0&1&0&1&1&1&1&1&1 \\ 0&1&1&1&1&1&1&1&1 \\ 0&1&1&1&1&1&1&1&1 \\ 0&1&1&1&1&1&1&1&1\end{smallmatrix} ,   \ 
\begin{smallmatrix}0&0&0&1&0&0&0&0&0 \\ 0&1&0&1&0&1&1&1&1 \\ 0&0&1&1&0&1&1&1&1 \\ 1&1&1&0&1&0&1&1&1 \\ 0&0&0&1&1&1&1&1&1 \\ 0&1&1&0&1&1&1&1&1 \\ 0&1&1&1&1&1&1&1&1 \\ 0&1&1&1&1&1&1&1&1 \\ 0&1&1&1&1&1&1&1&1\end{smallmatrix} ,   \ 
\begin{smallmatrix}0&0&0&0&1&0&0&0&0 \\ 0&1&0&0&1&1&1&1&1 \\ 0&0&1&0&1&1&1&1&1 \\ 0&0&0&1&1&1&1&1&1 \\ 1&1&1&1&0&0&1&1&1 \\ 0&1&1&1&0&1&1&1&1 \\ 0&1&1&1&1&1&1&1&1 \\ 0&1&1&1&1&1&1&1&1 \\ 0&1&1&1&1&1&1&1&1\end{smallmatrix} ,   \ 
\begin{smallmatrix}0&0&0&0&0&1&0&0&0 \\ 0&0&1&1&1&1&1&1&1 \\ 0&1&0&1&1&1&1&1&1 \\ 0&1&1&0&1&1&1&1&1 \\ 0&1&1&1&0&1&1&1&1 \\ 1&1&1&1&1&1&1&1&1 \\ 0&1&1&1&1&1&2&1&1 \\ 0&1&1&1&1&1&1&2&1 \\ 0&1&1&1&1&1&1&1&2\end{smallmatrix} ,   \ 
\begin{smallmatrix}0&0&0&0&0&0&1&0&0 \\ 0&1&1&1&1&1&1&1&1 \\ 0&1&1&1&1&1&1&1&1 \\ 0&1&1&1&1&1&1&1&1 \\ 0&1&1&1&1&1&1&1&1 \\ 0&1&1&1&1&1&2&1&1 \\ 1&1&1&1&1&2&1&1&2 \\ 0&1&1&1&1&1&1&2&2 \\ 0&1&1&1&1&1&2&2&1\end{smallmatrix} ,   \ 
\begin{smallmatrix}0&0&0&0&0&0&0&1&0 \\ 0&1&1&1&1&1&1&1&1 \\ 0&1&1&1&1&1&1&1&1 \\ 0&1&1&1&1&1&1&1&1 \\ 0&1&1&1&1&1&1&1&1 \\ 0&1&1&1&1&1&1&2&1 \\ 0&1&1&1&1&1&1&2&2 \\ 1&1&1&1&1&2&2&1&1 \\ 0&1&1&1&1&1&2&1&2\end{smallmatrix} ,   \ 
\begin{smallmatrix}0&0&0&0&0&0&0&0&1 \\ 0&1&1&1&1&1&1&1&1 \\ 0&1&1&1&1&1&1&1&1 \\ 0&1&1&1&1&1&1&1&1 \\ 0&1&1&1&1&1&1&1&1 \\ 0&1&1&1&1&1&1&1&2 \\ 0&1&1&1&1&1&2&2&1 \\ 0&1&1&1&1&1&2&1&2 \\ 1&1&1&1&1&2&1&2&1\end{smallmatrix} $$}
\begin{itemize}
\item Formal codegrees: $[7, 7, 7, 8, 9, 9, 9, 9, 504]$,
\item Property: simple, exotic, $3$-positive, isotype to $\Rep({\rm PSL}(2,8))$,
\item Categorification: open.
\end{itemize}
\item $\FPdim \ 1092$, type $[1,7,7,12,12,12,13,14,14]$, duality $[0,1,2,3,4,5,6,8,7]$, fusion data:
{\fontsize{7}{8}\selectfont $$ 
\begin{smallmatrix}1&0&0&0&0&0&0&0&0 \\ 0&1&0&0&0&0&0&0&0 \\ 0&0&1&0&0&0&0&0&0 \\ 0&0&0&1&0&0&0&0&0 \\ 0&0&0&0&1&0&0&0&0 \\ 0&0&0&0&0&1&0&0&0 \\ 0&0&0&0&0&0&1&0&0 \\ 0&0&0&0&0&0&0&1&0 \\ 0&0&0&0&0&0&0&0&1\end{smallmatrix} ,   \ 
\begin{smallmatrix}0&1&0&0&0&0&0&0&0 \\ 1&1&0&0&0&0&1&1&1 \\ 0&0&0&1&1&1&1&0&0 \\ 0&0&1&1&1&1&1&1&1 \\ 0&0&1&1&1&1&1&1&1 \\ 0&0&1&1&1&1&1&1&1 \\ 0&1&1&1&1&1&1&1&1 \\ 0&1&0&1&1&1&1&2&1 \\ 0&1&0&1&1&1&1&1&2\end{smallmatrix} ,   \ 
\begin{smallmatrix}0&0&1&0&0&0&0&0&0 \\ 0&0&0&1&1&1&1&0&0 \\ 1&0&1&0&0&0&1&1&1 \\ 0&1&0&1&1&1&1&1&1 \\ 0&1&0&1&1&1&1&1&1 \\ 0&1&0&1&1&1&1&1&1 \\ 0&1&1&1&1&1&1&1&1 \\ 0&0&1&1&1&1&1&2&1 \\ 0&0&1&1&1&1&1&1&2\end{smallmatrix} ,   \ 
\begin{smallmatrix}0&0&0&1&0&0&0&0&0 \\ 0&0&1&1&1&1&1&1&1 \\ 0&1&0&1&1&1&1&1&1 \\ 1&1&1&2&1&2&1&2&2 \\ 0&1&1&1&2&1&2&2&2 \\ 0&1&1&2&1&1&2&2&2 \\ 0&1&1&1&2&2&2&2&2 \\ 0&1&1&2&2&2&2&2&2 \\ 0&1&1&2&2&2&2&2&2\end{smallmatrix} ,   \ 
\begin{smallmatrix}0&0&0&0&1&0&0&0&0 \\ 0&0&1&1&1&1&1&1&1 \\ 0&1&0&1&1&1&1&1&1 \\ 0&1&1&1&2&1&2&2&2 \\ 1&1&1&2&2&1&1&2&2 \\ 0&1&1&1&1&2&2&2&2 \\ 0&1&1&2&1&2&2&2&2 \\ 0&1&1&2&2&2&2&2&2 \\ 0&1&1&2&2&2&2&2&2\end{smallmatrix} ,   \ 
\begin{smallmatrix}0&0&0&0&0&1&0&0&0 \\ 0&0&1&1&1&1&1&1&1 \\ 0&1&0&1&1&1&1&1&1 \\ 0&1&1&2&1&1&2&2&2 \\ 0&1&1&1&1&2&2&2&2 \\ 1&1&1&1&2&2&1&2&2 \\ 0&1&1&2&2&1&2&2&2 \\ 0&1&1&2&2&2&2&2&2 \\ 0&1&1&2&2&2&2&2&2\end{smallmatrix} ,   \ 
\begin{smallmatrix}0&0&0&0&0&0&1&0&0 \\ 0&1&1&1&1&1&1&1&1 \\ 0&1&1&1&1&1&1&1&1 \\ 0&1&1&1&2&2&2&2&2 \\ 0&1&1&2&1&2&2&2&2 \\ 0&1&1&2&2&1&2&2&2 \\ 1&1&1&2&2&2&2&2&2 \\ 0&1&1&2&2&2&2&3&2 \\ 0&1&1&2&2&2&2&2&3\end{smallmatrix} ,   \ 
\begin{smallmatrix}0&0&0&0&0&0&0&1&0 \\ 0&1&0&1&1&1&1&2&1 \\ 0&0&1&1&1&1&1&2&1 \\ 0&1&1&2&2&2&2&2&2 \\ 0&1&1&2&2&2&2&2&2 \\ 0&1&1&2&2&2&2&2&2 \\ 0&1&1&2&2&2&2&3&2 \\ 0&1&1&2&2&2&2&2&4 \\ 1&2&2&2&2&2&3&2&2\end{smallmatrix} ,   \ 
\begin{smallmatrix}0&0&0&0&0&0&0&0&1 \\ 0&1&0&1&1&1&1&1&2 \\ 0&0&1&1&1&1&1&1&2 \\ 0&1&1&2&2&2&2&2&2 \\ 0&1&1&2&2&2&2&2&2 \\ 0&1&1&2&2&2&2&2&2 \\ 0&1&1&2&2&2&2&2&3 \\ 1&2&2&2&2&2&3&2&2 \\ 0&1&1&2&2&2&2&4&2\end{smallmatrix} $$}
\begin{itemize}
\item Formal codegrees: $[4,7,7,7,12,12,13,13,1092]$,
\item Property: simple, exotic, $3$-positive, isotype to $\Rep({\rm PSL}(2,13))$,
\item Categorification: open.
\end{itemize}
\end{enumerate}

\section{Odd-dimensional case} \label{sec:Odd}

The results of~\cite{NS07} concerning Frobenius--Schur indicators impose strong constraints on odd-dimensional integral Grothendieck rings, influencing their type, duality, and formal codegrees (see~\S\ref{sub:MNSDRings}). These constraints naturally lead to the definition of MNSD Drinfeld rings (Definition \ref{def:MNSDring}). Guided by the table below, we then provide a complete classification of all MNSD Drinfeld rings. As a consequence, we obtain a classification of all odd-dimensional integral Grothendieck rings up to rank~$7$, thereby establishing Theorem~\ref{thm:mainodd}. The first example known to us of an odd-dimensional integral fusion category that is not Grothendieck equivalent to any Tannakian category occurs at rank~$27$; see~\S\ref{sub:NonTann}.

\begin{center}
\begin{tabular}{c|c|c|c|c}
\S & \textbf{Rank} & \textbf{Case} & \textbf{Bound on $\FPdim$} & \textbf{Number of Drinfeld rings} \\
\hline
\ref{sub:MNSDGeneralUpToR7} & $\le 7$ & All & All & $8$ \\
\ref{subsub:1FrobMNSDR9} & $9$ & $1$-Frobenius & All & $10$ \\
\ref{subsub:N1FrobMNSDR9} & $9$ & Non-perfect non-$1$-Frobenius & All & $2$ \\
\ref{subsub:N1FrobMNSDR9} & $9$ & Perfect non-$1$-Frobenius & $ < 389865 $ & $0$ \\
\ref{sub:MNSD1FrobR11} & $11$ & Non-perfect $1$-Frobenius & $\le 10^9$ & $24$
\end{tabular}
\end{center}

Our classification strategy begins with MNSD Egyptian fractions, following the approach of~\cite{ABPP}, with the key distinction that they sum to~$1$ and do not require squared denominators.

\subsection{MNSD Drinfeld rings} \label{sub:MNSDRings}



\begin{definition} \label{def:MNSDtype}
A sequence of positive integers $1 = m_1 \le \cdots \le m_r$ is called \emph{MNSD} if:
\begin{itemize}
    \item $r$ is odd, and each $m_i$ is odd;
    \item for all $j$, we have $m_{2j} = m_{2j+1}$.
\end{itemize}
Equivalently, such a sequence has the form
\[
(1, m_2, m_2, m_4, m_4, \dots, m_{r-1}, m_{r-1}).
\]
\end{definition}

\begin{theorem}[Corollary 8.2 in \cite{NS07}] \label{thm:OddMNSD}
Let $\mathcal{C}$ be an integral fusion category over $\mathbb{C}$ with odd Frobenius-Perron dimension. Then $\mathcal{C}$ is \emph{Maximally Non-Self Dual} (i.e., no simple object is self-dual except $\mathbf{1}$). In particular, its type is MNSD.
\end{theorem}

\begin{definition} \label{def:coMNSD}
A sequence of positive integers $n_1 \ge \cdots \ge n_r$ is called \emph{co-MNSD} if each $n_i$ divides $n_1$, and the sequence $(n_1 / n_i)$ is MNSD.
\end{definition}

\begin{definition} \label{def:EgyMNSD}
An Egyptian fraction $\sum_i 1/n_i = 1$ is called \emph{MNSD} if the sequence of positive integers $(n_i)$ is co-MNSD.
\end{definition}

\begin{proposition} \label{prop:MNSDFormalCodegrees}
Let $\mathcal{C}$ be an integral fusion category over $\mathbb{C}$ with odd Frobenius-Perron dimension. Then its formal codegrees (\S\ref{sub:formal}) form a co-MNSD sequence, and so make an MNSD Egyptian fraction.
\end{proposition}

\begin{proof}
Let $F\colon \mathcal{Z}(\mathcal{C}) \to \mathcal{C}$ denote the forgetful functor, and let $I\colon \mathcal{C} \to \mathcal{Z}(\mathcal{C})$ be its right adjoint. Suppose $A$ is a simple direct summand of $I(\mathbf{1})$ in $\mathcal{Z}(\mathcal{C})$, so that $\mathbf{1}$ appears as a subobject of $F(A)$ in~$\mathcal{C}$. Since $F$ is a tensor functor, it preserves duals, and thus $\mathbf{1}$ also appears in $F(A^*) = F(A)^*$. Therefore, $A^*$ is also a direct summand of $I(\mathbf{1})$.

As $\mathcal{C}$ is integral, it is pseudo-unitary and therefore spherical (see \cite{EGNO15}). By \cite[Theorem 2.13]{Os15}, the formal codegrees are given by $\FPdim(\mathcal{C}) / \FPdim(A)$, where $A$ runs over the simple summands of $I(\mathbf{1})$.

By \cite[Theorem 7.16.6]{EGNO15}, we have $\FPdim(\mathcal{Z}(\mathcal{C})) = \FPdim(\mathcal{C})^2$. Since $\FPdim(\mathcal{C})$ is odd, so is $\FPdim(\mathcal{Z}(\mathcal{C}))$. Hence, by Theorem~\ref{thm:OddMNSD}, the Drinfeld center $\mathcal{Z}(\mathcal{C})$ has MNSD type. The result is an immediate consequence of the preceding two paragraphs.
\end{proof}

\begin{definition} \label{def:MNSDring}
An integral Drinfeld ring is called \emph{MNSD} if:
\begin{itemize}
    \item it is Maximally Non-Self Dual (in particular, its type is MNSD);
    \item its formal codegrees form a co-MNSD sequence.
\end{itemize}
\end{definition}

It follows immediately from the above that the Grothendieck ring of any odd-dimensional integral fusion category over~$\mathbb{C}$ is an MNSD integral Drinfeld ring.


\begin{remark} \label{rk:ExoticRk4}
Here is the smallest example (from \cite{BP24}) of non-pointed, odd-dimensional, simple integral fusion rings:
\begin{itemize}
\item $\FPdim \ 7315$, type $[1,35,40,67]$, duality $[0,1,2,3]$, fusion data: 
$$ 
\left[\begin{matrix} 1&0&0&0 \\ 0&1&0&0 \\ 0&0&1&0 \\ 0&0&0&1\end{matrix} \right],  \ \left[ \begin{matrix}0&1&0&0 \\ 1&30&1&2 \\ 0&1&9&15 \\ 0&2&15&25\end{matrix} \right],  \ \left[ \begin{matrix}0&0&1&0 \\ 0&1&9&15 \\ 1&9&12&12 \\ 0&15&12&25\end{matrix} \right],  \ \left[ \begin{matrix}0&0&0&1 \\ 0&2&15&25 \\ 0&15&12&25 \\ 1&25&25&39\end{matrix}\right]
$$
\end{itemize}

However, we have not found any such examples that are MNSD Drinfeld or $1$-Frobenius.
\end{remark}

\begin{question}
Is there a non-pointed, odd-dimensional, simple integral fusion ring that is either MNSD Drinfeld or \( 1 \)-Frobenius?
\end{question}

Theorem~\ref{thm:OddNC2Intro} implies that, for Grothendieck rings of odd dimension, the noncommutative case can be safely ruled out at ranks less than~$21$. However, this exclusion does not hold \emph{a priori} for MNSD Drinfeld rings. Consequently, one must account for certain technical subtleties—similar to those addressed in \cite[\S\ref{sub:DivA}]{appendices}—which already cover ranks below~$9$. The verification at ranks~$9$ and~$11$ for MNSD Drinfeld rings follows the same method. Full details are provided in the file \verb|InvestR9R11MNSD.txt|, located in the \verb|Data/EgyptianFractionsDiv/Except| directory of~\cite{CodeData}.

\subsection{Up to rank 7} \label{sub:MNSDGeneralUpToR7}
This subsection is dedicated to prove Theorem \ref{thm:mainodd}. 

\subsubsection{Up to rank 5} There are four MNSD integral Drinfeld rings up to rank $5$, contained in \cite[\S\ref{sub:UpToRank5}]{appendices}, namely the Grothendieck rings of $\Rep(G)$, with $G = C_1, C_3, C_5$ and $C_7\rtimes C_3$. 

\subsubsection{Rank 7}
There are 13 MNSD Egyptian fractions of length~7 that satisfy the divisibility condition, as defined in Definition~\ref{def:EgyMNSD}. These are listed in the file \verb|EgyFracL7DivMNSD.txt|, located in the \verb|Data/EgyptianFractionsDiv| folder of~\cite{CodeData}. They correspond to 11 distinct global $\FPdim$ values:
\[
\{7, 15, 27, 35, 39, 55, 63, 147, 171, 315, 903\},
\]
yielding 11 possible MNSD types. Among these, only 4 types admit fusion rings---5 fusion rings in total---of which 4 are Drinfeld rings. A complete list of these Drinfeld rings, including their global $\FPdim$, type, duality, formal codegrees, and fusion rules, is provided in~\cite[\S\ref{sub:Rank7odd}]{appendices}. The corresponding Grothendieck rings are those of $\Rep(G)$ for \( G = C_7 \), \( C_{13} \rtimes C_3 \), \( C_{11} \rtimes C_5 \), and one excluded case detailed in \cite[\S\ref{sub:Rank7odd}(\ref{[1,1,1,3,3,21,21]})]{appendices}. Computational details and copy-pastable data are available in \verb|MNSDRank7.txt| within the \verb|Data/Odd| folder of~\cite{CodeData}.
 
\subsection{Rank 9} \label{sub:MNSDGeneralR9}
Starting from this rank, the classification becomes combinatorially complex. For this reason, we divide into two subsections: one dedicated to the 1-Frobenius case, and the other for discussing
 the rest.
\subsubsection{1-Frobenius} \label{subsub:1FrobMNSDR9}
There are 115 MNSD Egyptian fractions of length~$9$ satisfying the divisibility condition. These are listed in the file \verb|EgyFracL9DivMNSD.txt|, located in the \verb|Data/EgyptianFractionsDiv| folder of~\cite{CodeData}. They correspond to 76 distinct possible global $\FPdim$, which yield 17 possible 1-Frobenius MNSD types. Among these, only 8 types admit fusion rings—22 fusion rings in total—of which 10 are Drinfeld rings. A complete list of these Drinfeld rings, including their global $\FPdim$, type, duality, formal codegrees, and fusion data, is provided in~\cite[\S\ref{sub:Rank9MNSD1Frob}]{appendices}. Among them, exactly three are Grothendieck equivalent to Tannakian categories, namely $\Rep(G)$ for \( G = C_3^2 \), \( C_9 \), and \( C_{19} \rtimes C_3 \). One does not admit categorification, and the remaining six remain open. Computational details and copy-pastable data are available in the file \verb|MNSD1FrobRank9.txt| within the \verb|Data/Odd| folder of~\cite{CodeData}.

\subsubsection{Non-1-Frobenius} \label{subsub:N1FrobMNSDR9}
As observed in~\cite[\ref{sec:OddA}]{appendices}, all MNSD Drinfeld rings of rank up to~$7$ are 1-Frobenius. In contrast, there exist exactly two non-perfect, non-1-Frobenius MNSD Drinfeld rings of rank~$9$; see~\cite[\S\ref{sub:Rank9MNSDN1Frob}]{appendices}, as well as the copy-pastable data in the file \verb|N1FrobMNSDRank9.txt| located in the \verb|Data/Odd| directory of~\cite{CodeData}. In the perfect case, apart from the last three remaining global $\FPdim$s---$389865$, $544509$, and $1631721$—for which the classification is still incomplete, no further examples were found.

\subsection{Rank 11} \label{sub:MNSD1FrobR11} This subsection presents the classification of all 24 MNSD integral 1-Frobenius Drinfeld rings of rank~$11$ and global $\FPdim < 10^9$, except for six perfect types with $\FPdim > 10^6$ that remain open.

There are 2799 MNSD Egyptian fractions of length~$11$ satisfying the divisibility condition. These are listed in the file \verb|EgyFracL11DivMNSD.txt|, located in the \verb|Data/EgyptianFractionsDiv| folder of~\cite{CodeData}. They correspond to 1650 distinct possible global $\FPdim$, between $11$ and $5325028475403$. We restricted our analysis to those with $\FPdim < 10^9$, which excludes only the last 80. This yields 439 possible 1-Frobenius MNSD types. Among these, six perfect types with $\FPdim \ge 1789515$ are still open. Of the remaining types, only 25 admit fusion rings—491 fusion rings in total—of which 24 are Drinfeld rings. Among them, seven are Grothendieck equivalent to Tannakian categories, namely $\Rep(G)$ for $G = C_{11},\ C_9 \rtimes C_3,\ C_3^2 \rtimes C_3,\ C_5^2 \rtimes C_3,\ C_{31} \rtimes C_5,\ C_{19} \rtimes C_9,\ C_{29} \rtimes C_7$. One Drinfeld ring does not admit categorification, and the remaining ones remain open. Computational details and copy-pastable data are available in the file \verb|MNSD1FrobRank11.txt| within the \verb|Data/Odd| folder of~\cite{CodeData}.

\subsection{Non-Tannakian} \label{sub:NonTann}

This subsection concerns odd-dimensional integral fusion categories that are not Grothendieck equivalent to any Tannakian category. By Theorem~\ref{thm:mainodd}, no such example exists up to rank~$7$, and none has been found so far up to rank~$11$. However, there is a group-theoretical example at rank~$27$ with $\FPdim$~$75$ and type~$[[1,25],[5,2]]$, namely $\mathcal{C}(C_5^2 \rtimes C_3,1,C_5,1)$, as shown by the following computation:
\begin{verbatim}
gap> FindGroupSubgroup([1,1,1,1,1,1,1,1,1,1,1,1,1,1,1,1,1,1,1,1,1,1,1,1,1,5,5]);
[[2,"(C5 x C5) : C3",3,"C5",[[1,1,1,1,1,1,1,1,1,1,1,1,1,1,1,1,1,1,1,1,1,1,1,1,1,5,5],
         [0,4,3,2,1,20,24,23,22,21,15,19,18,17,16,10,14,13,12,11,5,9,8,7,6,26,25]]],
 [2,"(C5 x C5) : C3",4,"C5",[[1,1,1,1,1,1,1,1,1,1,1,1,1,1,1,1,1,1,1,1,1,1,1,1,1,5,5],
         [0,4,3,2,1,20,24,23,22,21,15,19,18,17,16,10,14,13,12,11,5,9,8,7,6,26,25]]]]
\end{verbatim}

\section{Noncommutative case} \label{sec:NC} 

We begin in~\S\ref{FuRingsNC} by presenting several constraints on the isomorphism classes of noncommutative complexified fusion rings, arising from Galois-theoretic considerations. In particular, we establish that rank~$6$ is the minimal possible rank for a noncommutative fusion ring. 

In~\S\ref{DrRingsNC}, we study Drinfeld rings and show—using Egyptian fraction techniques—that the group ring \(\mathbb{Z}S_3\) is the unique noncommutative integral Drinfeld ring of rank~6, and thus the only integral Grothendieck ring of that rank. We then extend the classification to all integral Grothendieck rings of rank up to~7, all noncommutative integral Drinfeld rings of rank up to~8, and finally to noncommutative integral \(1\)-Frobenius Drinfeld rings of rank~9 with \(\FPdim \le 10000\).

Finally, in~\S\ref{GrRingsNC}, leveraging results involving Frobenius--Schur indicators, we prove that any noncommutative, odd-dimensional integral Grothendieck ring must have rank at least~$21$. At this minimal rank, the sole example is the pointed fusion ring corresponding to the group~$C_7 \rtimes C_3$. We conclude with a discussion about the rank~$23$.

\begin{center}
\begin{tabular}{c|c|c|c|c}
\S & \textbf{Rank} & \textbf{Case} & \textbf{Bound on $\FPdim$} & \textbf{Number of Drinfeld rings} \\
\hline
\ref{subsub:NCR7} & $\le 7$ & NC Grothendieck & All & $3$ \\
\ref{subsub:NCR8} & $\le 8$ & NC Drinfeld & All & $29$ \\
\ref{subsub:NCR9} & $9$ & $1$-Frobenius NC & $\le 10000$ & $83$ \\
\ref{GrRingsNC} & $\le 21$ & Grothendieck + MNSD + NC & All & $1$
\end{tabular}
\end{center}

\subsection{Fusion rings} \label{FuRingsNC}

The exclusion of rank~$5$ in the following result is inspired by a MathOverflow response of Victor Ostrik~\cite{OsMO}, with additional insights provided by Noah Snyder. The next lemma is elementary.

\begin{lemma} \label{lem:galois}
Let \( P \in \mathbb{Q}[X] \), and let \(\sigma \in \operatorname{Aut}(\overline{\mathbb{Q}}/\mathbb{Q})\) be any Galois automorphism. Then \(\sigma\) permutes the roots of \(P\), preserving their multiplicities.
\end{lemma}

\begin{proposition} \label{prop:min6}
The minimal rank of a noncommutative fusion ring is~$6$.
\end{proposition}

\begin{proof}
Let \(\mathcal{R}\) be a fusion ring of rank \(r < 6\). Its complexified algebra \(\mathcal{R}_{\mathbb{C}} := \mathcal{R} \otimes_{\mathbb{Z}} \mathbb{C}\) is a finite-dimensional unital \(*\)-algebra and thus decomposes as a direct sum of matrix algebras over \(\mathbb{C}\). Since \(\mathcal{R}\) is assumed to be noncommutative, \(\mathcal{R}_{\mathbb{C}}\) must contain a summand isomorphic to \(M_n(\mathbb{C})\) for some \(n \geq 2\).

By Theorem~\ref{thm:FrobPer}, the Frobenius–Perron dimension defines a one-dimensional representation of \(\mathcal{R}_{\mathbb{C}}\). Since the rank \(r\) of \(\mathcal{R}\) equals the complex dimension of \(\mathcal{R}_{\mathbb{C}}\), we get:
\[
r \geq \dim_{\mathbb{C}}(\mathbb{C} \oplus M_2(\mathbb{C})) = 1 + 4 = 5.
\]

Now suppose \(r = 5\). From \S\ref{sub:formal}, let \(a = \FPdim(\mathcal{R})\), and let \(b\) denote the second formal codegree of \(\mathcal{R}\). These are positive algebraic integers satisfying
\[
\frac{1}{a} + \frac{2}{b} = 1.
\]
Rewriting this gives
\[
a = 1 + \frac{2}{b - 2}.
\]
In particular, \(b > 2\). Let \(P \in \mathbb{Z}[X]\) denote the characteristic polynomial of the multiplication matrix \(L_Z\) (see \S\ref{sub:formal}). Then the roots of \(P\) are \(a\) with multiplicity~1 and \(2b\) with multiplicity~4.

Suppose \(a = 2b\). Then \(a = \FPdim(\mathcal{R}) = 5 = r\), so \(\mathcal{R}\) is pointed, and corresponds to a nonabelian group of order~5, which does not exist. Thus \(a\) and \(2b\) are distinct.

By Lemma~\ref{lem:galois}, \(a\) and \(2b\) cannot be Galois conjugate, and hence must both be rational numbers. Since they are also algebraic integers, it follows that \(a, b \in \mathbb{Z}\). As \(b > 2\), we must have \(b \geq 3\). Therefore,
\[
\FPdim(\mathcal{R}) = a = 1 + \frac{2}{b - 2} \leq 3,
\]
which contradicts the fact that \(\FPdim(\mathcal{R}) \geq r = 5\). This contradiction shows that rank \(r = 5\) is impossible.

Finally, rank~6 is realized by the group ring of \(S_3\), completing the proof.
\end{proof}


\begin{lemma} \label{lem:excl}
The algebra \(\mathcal{R}_{\mathbb{C}}\) cannot be isomorphic to \(\mathbb{C}^{\oplus n} \oplus M_m(\mathbb{C})\) with \(n \leq 1\) and \(m \geq 2\).
\end{lemma}

\begin{proof}
The case \(n = 0\) is ruled out by the existence of the one-dimensional representation \(\FPdim\), so we must have \(n = 1\) and \(m \geq 2\). The rest of the proof follows similarly to that of Proposition~\ref{prop:min6}. Let \(a = \FPdim(\mathcal{R})\) and \(b\) be the formal codegree associated to the matrix block \(M_m(\mathbb{C})\). If \(a = mb\), then \(\mathcal{R}\) is pointed. But there is no nonabelian group \(G\) such that \(\Rep(G)\) has rank \(2\). Therefore, \(a \ne mb\), and the contradiction arises from the following inequality:
\[
1 + m^2 = \operatorname{rank}(\mathcal{R}) \leq \FPdim(\mathcal{R}) = 1 + \frac{m}{b - m} \leq 1 + m. \qedhere
\]
\end{proof}

\begin{lemma} \label{lem:NCmin}
Let $\mathcal{R}$ be a noncommutative fusion ring of rank $<9$. Then $\mathcal{R}_{\mathbb{C}}$ is isomorphic to $\mathbb{C}^{\oplus n} \oplus M_2(\mathbb{C})$ with $n \in \{2,3,4\}$.
\end{lemma}
\begin{proof}
Immediate from Lemma \ref{lem:excl}.
\end{proof}

\subsection{Drinfeld Rings} \label{DrRingsNC}
The results in this subsection apply to fusion rings that satisfy the Drinfeld condition (see~\S\ref{sub:Drin}).


\subsubsection{Rank 6 and proof of Theorem \ref{thm:IntroNCDrIntRank6}}

\begin{proof}
By Lemma~\ref{lem:NCmin}, the complexified ring $\mathcal{R}_{\mathbb{C}}$ must be isomorphic to $\mathbb{C}^{\oplus 2} \oplus M_2(\mathbb{C})$. Under the assumption that $\mathcal{R}$ is an integral Drinfeld ring, we are reduced to classifying all Egyptian fractions of the form  
\[
\frac{1}{a} + \frac{1}{b} + \frac{1}{c} + \frac{1}{c} = 1,
\]
with $b, c \mid a$ and $a \geq 6$. The complete list of such solutions, under these divisibility constraints, is provided in the file \verb|EgyFracL4Div.txt| in the \verb|Data/EgyptianFractionsDiv| directory of~\cite{CodeData}. Exactly four solutions exist:
\[
(2,5,5,10), \quad (2,6,6,6), \quad (3,3,6,6), \quad (3,3,4,12).
\]
These correspond to global $\FPdims$ in the set $\{6,10,12\}$. Among these, only two types of rank~$6$ are possible:
\[
[1,1,1,1,1,1] \quad \text{and} \quad [1,1,1,1,2,2].
\]
Both types admit realizations as fusion rings, yielding six examples in total, all of which satisfy the Drinfeld condition. However, only one of them is noncommutative: the group ring $\mathbb{Z}S_3$. For computational verification, see \verb|InvestNCRank6.txt| in the \verb|Data/Noncommutative| folder of~\cite{CodeData}.
\end{proof}

\begin{corollary} \label{cor:catNCrank6}
Up to Grothendieck equivalence, $\VVec(S_3)$ is the only integral fusion category over~$\mathbb{C}$ of rank at most~$6$ whose Grothendieck ring is noncommutative. There are precisely six such categories, all of the form $\VVec(S_3, \omega)$.
\end{corollary}

Extending Theorem~\ref{thm:IntroNCDrIntRank6} beyond the integral case appears to be a challenging problem. The database in~\cite{VS23} already lists thirteen noncommutative fusion rings of rank~$6$, with multiplicities reaching as high as~$8$ (the classification is complete up to multiplicity~$4$). All of them are Drinfeld rings. Among these, only one is integral—namely, $\mathbb{Z}S_3$—and three are simple fusion rings. This leads to the following open questions, now stated without assuming the Drinfeld condition:

\begin{question} \label{q:IntR6}
Is $\mathbb{Z}S_3$ the only noncommutative integral fusion ring of rank~$6$?
\end{question}

\begin{question} \label{q:SimpleR6}
Are there infinitely many noncommutative simple fusion rings of rank~$6$?
\end{question}


There exists an infinite family $(\mathcal{R}_n)_{n \geq 0}$ of non-simple noncommutative Drinfeld rings of rank 6, with $\mathcal{R}_0 = \mathbb{Z}S_3$:
$$\left[ \begin{smallmatrix} 1&0&0&0&0&0 \\ 0&1&0&0&0&0 \\ 0&0&1&0&0&0 \\ 0&0&0&1&0&0 \\ 0&0&0&0&1&0 \\ 0&0&0&0&0&1 \end{smallmatrix} \right],\ 
 \left[ \begin{smallmatrix} 0&1&0&0&0&0 \\ 0&0&1&0&0&0 \\ 1&0&0&0&0&0 \\ 0&0&0&0&1&0 \\ 0&0&0&0&0&1 \\ 0&0&0&1&0&0 \end{smallmatrix} \right],\ 
 \left[ \begin{smallmatrix} 0&0&1&0&0&0 \\ 1&0&0&0&0&0 \\ 0&1&0&0&0&0 \\ 0&0&0&0&0&1 \\ 0&0&0&1&0&0 \\ 0&0&0&0&1&0 \end{smallmatrix} \right],\ 
 \left[ \begin{smallmatrix} 0&0&0&1&0&0 \\ 0&0&0&0&0&1 \\ 0&0&0&0&1&0 \\ 1&0&0&{\rm n}&{\rm n}&{\rm n} \\ 0&0&1&{\rm n}&{\rm n}&{\rm n} \\ 0&1&0&{\rm n}&{\rm n}&{\rm n} \end{smallmatrix} \right],\ 
 \left[ \begin{smallmatrix} 0&0&0&0&1&0 \\ 0&0&0&1&0&0 \\ 0&0&0&0&0&1 \\ 0&1&0&{\rm n}&{\rm n}&{\rm n} \\ 1&0&0&{\rm n}&{\rm n}&{\rm n} \\ 0&0&1&{\rm n}&{\rm n}&{\rm n} \end{smallmatrix} \right],\ 
 \left[ \begin{smallmatrix} 0&0&0&0&0&1 \\ 0&0&0&0&1&0 \\ 0&0&0&1&0&0 \\ 0&0&1&{\rm n}&{\rm n}&{\rm n} \\ 0&1&0&{\rm n}&{\rm n}&{\rm n} \\ 1&0&0&{\rm n}&{\rm n}&{\rm n} \end{smallmatrix} \right]$$
The above fusion data can be interpreted as follows: Let $(b_g)_{g \in S_3}$ denote the basis, and let $g_1$, $g_2$, and $g_3$ be the three elements of $S_3$ of order 2. The fusion rules are given by:
\[
b_g b_h = b_{gh} + n \delta_{\ord(g), 2} \delta_{\ord(h), 2} \left( b_{g_1} + b_{g_2} + b_{g_3} \right).
\]
We compute that
$
\FPdim(\mathcal{R}_n) = 9 n \alpha_n + 6,
$
its type is 
$
[1, 1, 1, \alpha_n, \alpha_n, \alpha_n],
$
and its formal codegrees are 
\[
\bigl[3_2,\; 27 n^2 - 9 n \alpha_n + 6,\; 9 n \alpha_n + 6\bigr],
\]
where 
$
\alpha_n = \frac{3 n + \sqrt{9 n^2 + 4}}{2}.
$
From this, we deduce that $\mathcal{R}_n$ is a Drinfeld ring, since
\[
\frac{9 n \alpha_n + 6}{27 n^2 - 9 n \alpha_n + 6} = \alpha_n^2.
\]


\subsubsection{Rank 7 and proof of Theorem~\ref{thm:IntroNCGrIntRank7}} \label{subsub:NCR7}

\begin{proposition} \label{prop:NCDrIntRank7}
There are exactly three noncommutative integral Drinfeld rings of rank $7$:
\begin{itemize}
\item $\FPdim \ 24$, type $[1,1,1,2,2,2,3]$, duality $[0,2,1,3,4,5,6]$, formal codegrees $[3_2 , 4, 24, 24]$,
\item $\FPdim \ 42$, type $[1,1,1,1,1,1,6]$, duality $[0,1,2,3,5,4,6]$, formal codegrees $[3_2 , 6, 7, 42]$,
\item $\FPdim \ 60$, type $[1,1,1,3,4,4,4]$, duality $[0,2,1,3,4,5,6]$, formal codegrees $[3_2 , 4, 15, 60]$,
\end{itemize}
where the notation $3_2$ is explained in \cite[\ref{sec:NCA}]{appendices}. The fusion data are available in \cite[\S\ref{sub:NCRank7}]{appendices}.
\end{proposition}
\begin{proof}
By Lemma \ref{lem:NCmin}, the complexified fusion ring $\mathcal{R}_{\mathbb{C}}$ must be isomorphic to $\mathbb{C}^{\oplus 3} \oplus M_2(\mathbb{C})$. Under the assumption that the ring is integral and Drinfeld, the problem reduces to classifying all Egyptian fractions of the form
$$
\frac{1}{a} + \frac{1}{b} + \frac{1}{c} + \frac{1}{d} + \frac{1}{d} = 1,
$$
where $a \geq 7$ and $b$, $c$, $d$ divide $a$. Using the classification of length-5 Egyptian fractions with divisibility constraints—available in the file \verb|EgyFracL5Div.txt| in the \verb|Data/EgyptianFractionsDiv| directory of \cite{CodeData}—exactly 47 such solutions exist. These correspond to the following 22 possible values of $\FPdim(\mathcal{R})$:
$$
\{8, 9, 10, 12, 14, 15, 18, 20, 21, 24, 30, 36, 42, 45, 48, 60, 70, 78, 84, 110, 120, 156\}.
$$
For these $\FPdims$, there are 83 possible rank-7 fusion ring types. Among them, only 11 support fusion rings, yielding 44 distinct rings in total. Of these, 20 are Drinfeld fusion rings, among which 3 are noncommutative. For computational details, see the file \verb|InvestNCRank7.txt| in the \verb|Data/Noncommutative| directory of \cite{CodeData}.
\end{proof}

The proof of Theorem~\ref{thm:IntroNCGrIntRank7} follows from Theorem~\ref{thm:IntroNCDrIntRank6}, Proposition~\ref{prop:NCDrIntRank7}, and the explicit models and exclusions presented in~\cite[\S\ref{sub:NCRank7}]{appendices}. Note that all noncommutative integral Drinfeld rings up to rank $7$ are $1$-Frobenius; however, this no longer holds from rank $8$ onwards.

\subsubsection{Rank 8} \label{subsub:NCR8}

\begin{proposition} \label{prop:NCDrIntRank8}
There are exactly $25$ noncommutative integral Drinfeld rings of rank $8$. Complete fusion data is provided in \cite[\S\ref{sub:NCRank8}]{appendices}. Precisely five of these Drinfeld rings are not $1$-Frobenius.
\end{proposition}
\begin{proof}
By Lemma \ref{lem:NCmin}, the complexification $\mathcal{R}_{\mathbb{C}}$ must be isomorphic to $\mathbb{C}^{\oplus 4} \oplus M_2(\mathbb{C})$. Under the integrality assumption for Drinfeld rings, this reduces the classification problem to determining all Egyptian fractions of the form
$$
\frac{1}{a} + \frac{1}{b} + \frac{1}{c} + \frac{1}{d} + \frac{1}{e} + \frac{1}{e} = 1,
$$
where $b, c, d, e$ divide $a \geq 8$. A complete enumeration of such length-$6$ Egyptian fractions under these divisibility constraints—available in the file \verb|EgyFracL6Div.txt| in the \verb|Data/EgyptianFractionsDiv| directory of \cite{CodeData}—yields exactly $524$ valid solutions. These give rise to $143$ distinct global $\FPdim$ values, ranging from $8$ to $24492$. For these values, a total of $6539044$ rank-$8$ types are theoretically possible.

Imposing the $1$-Frobenius condition reduces this number drastically to $2484$ types. Among these, only $132$ support fusion ring structures, resulting in $3682$ fusion rings in total, of which $338$ are Drinfeld. Out of these, $20$ are noncommutative. This entire computation can be completed in under 10 minutes on a standard laptop.

Without the $1$-Frobenius assumption, the same classification method applies but requires significantly more computational effort, taking several weeks on a HPC. In this more general case, we obtain exactly $25$ noncommutative integral Drinfeld rings of rank $8$. Further computational details can be found in \verb|InvestNCRank8.txt|, located in the \verb|Data/Noncommutative| directory of \cite{CodeData}.
\end{proof}

Among the $25$ Drinfeld rings mentioned in Proposition~\ref{prop:NCDrIntRank8}, $5$ have already been excluded, and group-theoretical models have been identified for another $5$. The remaining $15$ cases remain open.

\subsubsection{Rank 9} \label{subsub:NCR9}
Without going into detail—and \emph{without} pushing \Normaliz{} to its limits—we classified all $83$ integral $1$-Frobenius noncommutative Drinfeld rings of rank~$9$ with $\FPdim \le 10000$. Complete computational details and copy-pastable data can be found in the file \verb|1FrobR9NCd10000.txt|, located in the \verb|Data/Noncommutative| directory of~\cite{CodeData}.

\subsection{Grothendieck rings}  \label{GrRingsNC}
 \begin{proposition} \label{prop:DimDet}  
Let $\mathcal{C}$ be an integral fusion category over $\mathbb{C}$, and let $\mathcal{R}$ be its Grothendieck ring. The integers  
\[
\FPdim(\mathcal{R}) \quad \text{and} \quad \prod_{V \in {\rm Irr}(\mathcal{R}_{\mathbb{C}})} n_V f_V,  
\]  
where $n_V$ is the dimension of $V$ and $f_V$ its formal codegree (see \S\ref{sub:formal}), have the same prime divisors.  
\end{proposition}  
\begin{proof}  
An integral fusion category is pseudo-unitary and therefore spherical \cite{EGNO15}. By \cite[Theorem 4.3]{WLL21}, the determinant of the left multiplication matrix of $Z$ (as defined in \S\ref{sub:formal}) has the same prime divisors as the Frobenius-Schur exponent of $\mathcal{C}$, which in turn shares its prime divisors with $\FPdim(\mathcal{R})$ by \cite[Theorem 8.4]{NS07}.  

Since $ Z $ is diagonalizable with eigenvalues $ n_V f_V $, each having multiplicity $ n_V^2 $, its determinant is
\[
\prod_{V \in {\rm Irr}(\mathcal{R}_{\mathbb{C}})} (n_V f_V)^{n_V^2}.  
\]  
Finally, because $f_V$ is an integer dividing $\FPdim(\mathcal{R})$ (see \S\ref{sub:formal}), we can ignore the exponent ${n_V^2}$ when considering prime divisors, proving the result.  
\end{proof}

It follows directly from Proposition \ref{prop:DimDet} that:  

\begin{corollary}   \label{cor:OddNC}
Let $ \mathcal{C} $ be an integral fusion category over $ \mathbb{C} $, and let $ \mathcal{R} $ be its Grothendieck ring. If a prime $ p $ does not divide $ \FPdim(\mathcal{R}) $, then it also does not divide $ n_V $ for all $ V \in {\rm Irr}(\mathcal{R}_{\mathbb{C}}) $. In particular, if $ \FPdim(\mathcal{R})$ is odd, then every irreducible representation of $ \mathcal{R}_{\mathbb{C}} $ is also odd-dimensional. Moreover, if $ \mathcal{R} $ is noncommutative, there exists some $ V $ with $ n_V \geq 3 $. 
\end{corollary} 

\begin{corollary} \label{cor:Odd}  
An integral Grothendieck ring of odd global $\FPdim$ has odd rank.  
\end{corollary}

\begin{proof}
The claim follows by reducing modulo $2$ the identity from~\S\ref{sub:formal}:
\[
\sum_{V \in \mathrm{Irr}(\mathcal{R}_{\mathbb{C}})} \sum_{i=1}^{n_V^2} \frac{1}{n_V f_V} = 1,
\]
combined with the fact that both $n_V$ and $f_V$ are odd by Proposition~\ref{prop:DimDet}, and that the rank equals $\sum_V n_V^2$.
\end{proof}

Following the notation used in the proof of Proposition~\ref{prop:MNSDFormalCodegrees}, the lemma below is an immediate consequence of \cite[Corollary 2.16]{Os15}:

\begin{lemma} \label{lem:MultI(1)}
Let \( V \) be an irreducible representation of the complexified integral Grothendieck ring \( \mathcal{K}(\mathcal{C})_{\mathbb{C}} \). Then the corresponding simple object \( A_V \) in \( \mathcal{Z}(\mathcal{C}) \) appears as a direct summand of \( I(\one) \) with multiplicity \( \dim(V) \).
\end{lemma}


\subsubsection{Proof of Theorem \ref{thm:OddNC2Intro}} \label{subsub:ProofOddNC2Intro}

\begin{proof}  
The smallest order of a noncommutative finite group of odd order is \( 21 \), uniquely realized by \( G = C_7 \rtimes C_3 \). Consequently, \( \VVec(G) \) yields a noncommutative integral Grothendieck ring of odd dimension and rank \( 21 \). 

We now show that there exists no non-pointed, noncommutative, odd-dimensional integral Grothendieck ring of rank less than or equal to \( 21 \). 

By Proposition~\ref{prop:MNSDFormalCodegrees}, Corollary~\ref{cor:OddNC}, and Lemma~\ref{lem:MultI(1)}, there must exist some irreducible representation \( V \) with \( n_V \geq 3 \) and \( V^* \not\simeq V \). Additionally, we have \( n_{\FPdim} = 1 \). Hence, the rank must be at least 
$
1 + 3^2 + 3^2 = 19.
$ 
Corollary~\ref{cor:Odd} excludes ranks \( 20, 22 \), leaving only ranks \( 19 \) and \( 21 \) for consideration.

To exclude rank \( 19 \), consider an Egyptian fraction of the form
\[
\frac{1}{a} + \frac{3}{b} + \frac{3}{b} = 1,
\]
where \( b \mid a \geq 19 \). Rewriting, we obtain \( 6a + b = ab \), which implies \( b \equiv 0 \mod a \), i.e., \( a \mid b \). Thus, \( a = b = 7 \), contradicting the assumption \( a \geq 19 \).

Now consider rank \( 21 \). The relevant Egyptian fraction has the form
\[
\frac{1}{a} + \frac{1}{b} + \frac{1}{b} + \frac{3}{c} + \frac{3}{c} = 1,
\]
with \( b, c \mid a \geq 21 \). If \( b = c \), then as before, we get \( a = b = c = 9 \), again contradicting \( a \geq 21 \). Hence, \( b \neq c \).

According to the classification of length-9 MNSD Egyptian fractions under divisibility constraints—recorded in the file \verb|EgyFracL9DivMNSD.txt| located in the \verb|Data/EgyptianFractionsDiv| directory of \cite{CodeData}—the only solutions are:
\[
(a,b,c) = (21,3,21),\, (21,21,7),\, (57,3,19),\, (105,15,7).
\]
In the non-pointed case, the global $\FPdim$ must be \( 57 = 3 \times 19 \) or \( 105 = 3 \times 5 \times 7 \), both of which are square-free and have at most three prime divisors. By~\cite[Theorem 9.2]{ENO11}, we can reduce to the group-theoretical case, and hence by~\cite[Theorem 1.5]{ENO11}, to the $1$-Frobenius case. However, there exists no MNSD $1$-Frobenius type of rank~21 with $\FPdim = 57$ or $105$, as confirmed by the following computation. The corresponding \SageMath{} code is available in the \verb|Code/SageMath| directory of~\cite{CodeData}.
\begin{verbatim}
sage: %attach TypesFinder1Frob.spyx
sage: [TypesFinderMNSD(i,21) for i in [57,105]]
[[], []]
\end{verbatim}
This concludes the proof.
\end{proof}

%

\subsubsection*{Rank 23}

Let us now discuss the rank 23. The relevant Egyptian fraction is of the form:
\[
\frac{1}{a} + \frac{1}{b} + \frac{1}{b} + \frac{1}{c} + \frac{1}{c} + \frac{3}{d} + \frac{3}{d} = 1,
\]
with \( b, c, d \mid a \geq 23 \). According to the classification of length-11 MNSD Egyptian fractions under divisibility constraints, as recorded in the file \verb|EgyFracL11DivMNSD.txt|, located in the \verb|Data/EgyptianFractionsDiv| directory of \cite{CodeData}, there are 48 such Egyptian fractions. This implies that the global $\FPdim$ belongs to the set
\begin{align*}
\{ &27, 35, 39, 55, 63, 75, 99, 119, 147, 155, 171, 175, 195, 203, 315, 399, 495, 595, 735, 903, 1155, 1575, \\
&2035, 2223, 2667, 3255, 4515, 6555, 7455, 22155 \}.
\end{align*}
While there are millions of possible MNSD types, completing the classification in general may prove too complex. However, restricting the analysis to the 1-Frobenius case reduces the number of possible MNSD types to just 118, corresponding to only five global $\FPdim$ values, namely
\[
\{39, 119, 903, 1575, 3255\}.
\]
The first three values, \(39 = 3 \times 13\), \(119 = 7 \times 17\), and \(903 = 3 \times 7 \times 43\), are square-free with at most three prime factors. As in the proof of \S\ref{subsub:ProofOddNC2Intro}, this reduces us to the group-theoretical, and hence the MNSD $1$-Frobenius case:
\begin{verbatim}
sage: %attach TypesFinder1Frob.spyx
sage: for i in [39, 119, 903]:
....:     print(i,TypesFinderMNSD(i,23))
39  [[1,1,1,1,1,1,1,1,1,1,1,1,1,1,1,1,1,1,1,1,1,3,3]]
119 [[1,1,1,1,1,1,1,1,1,1,1,1,1,1,1,1,1,1,1,1,1,7,7]]
903 [[1,1,1,1,1,1,1,1,1,1,1,1,1,1,1,1,1,1,1,1,1,21,21],
     [1,1,1,3,3,7,7,7,7,7,7,7,7,7,7,7,7,7,7,7,7,7,7]]
\end{verbatim}
Let \texttt{L} be the list of the four types mentioned above. We now apply the method described in \S\ref{sub:r7d903}:
\begin{verbatim}
gap> Read("GroupTheoretical.gap");
gap> for l in L do A:=FindGroupSubgroup(l);; if Length(A)>0 then Print(l,A); fi; od;
[[1,"C43 : C21",4,"C21",[[1,1,1,1,1,1,1,1,1,1,1,1,1,1,1,1,1,1,1,1,1,21,21],
          [0,6,5,4,3,2,1,14,20,19,18,17,16,15,7,13,12,11,10,9,8,22,21]]],
 [1,"C43 : C21",8,"C43 : C21",[[1,1,1,1,1,1,1,1,1,1,1,1,1,1,1,1,1,1,1,1,1,21,21],
          [0,2,1,18,20,19,15,17,16,12,14,13,9,11,10,6,8,7,3,5,4,22,21]]]]
\end{verbatim}
We deduce that all types are excluded except the third one, for which there are two possible models: \( \Rep(G) \) and \( \mathcal{C}(G,1,H,1) \), where \( G = C_{43} \rtimes C_{21} \) and \( H = C_{21} \). The Grothendieck ring of the former is commutative. For the latter, we verified in small examples that \( \Rep(K \rtimes H) \) and \( \mathcal{C}(K \rtimes H,1,H,1) \) always share the same type. This suggests that they may always be equivalent—at least when \( K \) and \( H \) are cyclic—which would eliminate the third type as a candidate in the noncommutative case.

Next, consider MNSD types of rank~23 with global \( \FPdim = 1575 = 3^2 \times 5^2 \times 7 \) and \( \FPdim = 3255 = 3 \times 5 \times 7 \times 31 \). There are 15810 types in the first case and 214752 in the second:
\begin{verbatim}
sage: %attach TypesFinder.spyx
sage: [[i,len(TypesFinderMNSD(i, 23))] for i in [1575, 3255]]
[[1575, 15810], [3255, 214752]]
\end{verbatim}

However, restricting to the $1$-Frobenius case drastically reduces the counts to 73 and 41, respectively:
\begin{verbatim}
sage: %attach TypesFinder1Frob.spyx
sage: [[i,len(TypesFinderMNSD(i, 23))] for i in [1575, 3255]]
[[1575, 73], [3255, 41]]
\end{verbatim}

Roughly one third of the remaining types are perfect, making a brute-force classification still unattainable. In both cases, a reduction to group-theoretical categories is ruled out. Furthermore, using \GAP{} (with the \textsf{SmallGrp}~1.5.1 package), we verified that no category of the form \( \mathcal{C}(G,1,H,1) \) exists with rank~23 and global \( \FPdim = 1575 \) or \( 3255 \).

\begin{verbatim}
gap> Read("GroupTheoretical.gap");
gap> for i in [1575,3255] do Print(i,FindGroupSubgroupOrderRank(i,23)); od;
1575[  ]3255[  ]
\end{verbatim}


\section*{Acknowledgments} 
We thank Shlomo Gelaki and Dmitri Nikshych for valuable discussions on group-theoretical fusion categories.
Sebastien Palcoux is supported by NSFC (Grant no. 12471031). The high-performance computing cluster at the University of Osnabrück, which played a crucial role in our computations, was funded by DFG grant 456666331.

\vspace*{.5cm}

\noindent \textbf{Availability of data and materials.} Data for the computations in this paper are available on reasonable request from the authors. The softwares used for the computations can be downloaded from the URLs listed in the references.

\vspace*{.15cm}

\noindent \textbf{Conflict of interest statement.} On behalf of all authors, the corresponding author declares that there are no conflicts of interest.



\end{document}